\documentclass[11pt, oneside]{article}
\usepackage[margin=1in]{geometry}

\usepackage{amsmath, amsfonts, amssymb, amsthm}
\usepackage{bm}
\usepackage{bbm} 
\usepackage[shortlabels]{enumitem}
\usepackage{fancyhdr}
\usepackage[T1]{fontenc}
\usepackage{graphicx}
\usepackage{listings}
\usepackage{sectsty}
\usepackage{setspace}
\usepackage{titlesec}
\usepackage[tracking]{microtype}
\usepackage{array, booktabs, ragged2e}
\usepackage[bookmarks,hypertexnames=false,debug,linktocpage=true,hidelinks]{hyperref}

\usepackage{color}

\allowdisplaybreaks 
\renewcommand{\thefootnote}{\arabic{footnote}}
\makeatletter
\renewcommand*{\@fnsymbol}[1]{\@arabic{#1}}
\makeatother

\hypersetup{ colorlinks, linktoc=all, linkcolor={black}, citecolor={black},
urlcolor={black} }

\titleformat{\chapter}
  {\color{black}\normalfont\LARGE\bfseries} {\thechapter}{1em} {}[]

\sectionfont{\color{black}} \subsectionfont{\color{black}}

\graphicspath{{./figures/}}

\newcommand\blfootnote[1]{%
  \begingroup
  \renewcommand\thefootnote{}\footnote{#1}%
  \addtocounter{footnote}{0}%
  \endgroup
}

\newcommand{\E}{\mathbb{E}}

\newcommand{\mbb}[1]{\mathbb{#1}}
\newcommand{\mf}[1]{\mathfrak{#1}}
\newcommand{\1}[1]{\mathbbm{1}_{\{#1\}}}
\newcommand{\2}[1]{\mathbbm{1}_{#1}}

\newcommand{\mc}[1]{\mathcal{#1}}

\newcommand*{\wt}[1]{\widetilde{#1}}

\newcommand{\ccmfis}[0]{\mc{I}^{\text{cCMF},i}_t}
\newcommand{\dcmfis}[0]{\mc{I}^{\text{dCMF},i}_t}
\newcommand{\wh}[1]{\widehat{#1}}
\newcommand{\setsep}{\;\big|\;}
\newcommand{\bigsetsep}{\;\bigg|\;}
\newcommand{\ess}{\bar{\mbb{X}}^N}
\newcommand{\eas}{\bar{\mbb{U}}^N}
\newcommand{\mss}{\widehat{\mbb{X}}^N}
\newcommand{\mas}{\widehat{\mbb{U}}^N}

 \DeclareMathOperator{\as}{a.s.}

\DeclareMathOperator{\iid}{i.i.d.}
 
\DeclareMathOperator{\DM}{DM} 
 
\DeclareMathOperator{\rep}{R} \DeclareMathOperator{\cex}{C-Ex}
\DeclareMathOperator{\csym}{C-Sym}

\newtheorem{theorem}{Theorem}[section]
\newtheorem{proposition}[theorem]{Proposition}
\newtheorem{corollary}[theorem]{Corollary}
\newtheorem{lemma}[theorem]{Lemma}

\theoremstyle{definition}
\newtheorem{definition}[theorem]{Definition}

\newtheorem{assumption}[theorem]{Assumption}

\begin{document}
\title{Mean-Field Systems with Heterogeneous Subteams: 
Optimality of Cluster-Symmetric Independent Policies and Equivalence with Decentralized McKean-Vlasov Control of Cluster-Representative Agents}
\author{Connor Braun$^1$\blfootnote{\llap{\textsuperscript{1}}Department of Mathematics and Statistics, Queen's University, Canada \texttt{connor.braun@queensu.ca}}, %
 Sina Sanjari$^2$\blfootnote{\llap{\textsuperscript{2}}Department of Mathematics and Computer Science, Royal Military College of Canada \texttt{sanjari@rmc.ca}}, %
 Naci Saldi$^3$\blfootnote{\llap{\textsuperscript{3}}Department of Mathematics, Bilkent University, Turkey \texttt{naci.saldi@bilkent.edu.tr}}, %
 Gunnar Blohm$^4$\blfootnote{\llap{\textsuperscript{4}}Centre for Neuroscience Studies, Queen's University, Canada \texttt{gunnar.blohm@queensu.ca}} %
 and Serdar Y\"uksel$^{5\,*}$\blfootnote{\llap{\textsuperscript{5}}Department of Mathematics and Statistics, Queen's University, Canada \texttt{yuksel@queensu.ca}\\\textsuperscript{*}This research is supported in part by the Natural Sciences and Engineering Research Council (NSERC) of Canada.}}
\maketitle
\setcounter{section}{0}
{\small
\begin{center}
  \begin{abstract}
  \noindent Across science and engineering, mean-field methods have been a
  powerful and versatile approach for the analysis of systems of many
  interacting elements. However, common arguments used to characterize an
  infinite population limit can be quite restrictive from a modeling perspective
  by requiring that all agents be identical (i.e. symmetric, or homogeneous). In
  this paper, we consider large interactive particle systems under agent
  heterogeneity for a class of discrete-time teams composed of finitely many
  species of agents, grouped  into symmetric subteams, called clusters. In
  particular, for the class of discounted, partially exchangeable cost criteria
  considered, we establish the optimality of centralized joint policies which
  are exchangeable within each cluster and depend on the agent ensemble only up
  to the state empirical distribution over each cluster. Following this, a
  generalization of De Finetti's theorem is used to demonstrate the
  subsequential convergence of these optimal policies to one which is
  decentralized (depending on only the local state and distribution over each
  cluster) and symmetric within each subteam as the population size approaches
  infinity. This solution is shown to induce a sequence of asymptotically
  optimal policies for the finite population problems which retain their
  structure and decentralization. Furthermore, our analysis justifies the
  optimality of a decentralized McKean-Vlasov team representation involving
  coupled representative agents for each of the clusters, and establishes a
  verification theorem/value iterations for the mean-field limit. In this way,
  we provide an avenue for analyzing complex, cooperative systems with finite
  heterogeneity and set the stage for further research on learning algorithms.
  
  \end{abstract}
\end{center}
}
\section{Introduction}
In various domains of science and engineering, physical systems admit a
decomposition into a collection of discrete decision makers (DMs) or agents.
Modelled DMs are equipped with local dynamics which may depend on the joint
action of the entire population. In turn, agents strive to select actions
shaping the joint state evolution in a manner which either (i) collectively
optimizes a shared objective, or (ii) optimizes a local objective which
disagrees with at least one other agent. In the former case the system is called
a stochastic team, and in the latter noncooperative, or a stochastic game.
Alternatively, a team can be regarded as a degenerate instance of a game where
agent interests are maximally aligned. In either case, agents may share all
available information, or else may act according to only their local, private
knowledge of the system. These modes of decision-making are called centralized
and decentralized, respectively, with the latter being significantly more
challenging to study, in general.\\[5pt]
\indent Complex strategic environments involving very large, weakly coupled
(i.e., when interaction strengths scale inversely with the number of impinging
elements) agent populations arise naturally in many applications. For these, a
powerful means of facilitating optimization/analysis is to pass to a suitable
mean-field limit representing the infinite population, where the problem often
simplifies to one involving a distributional flow summarizing the behavior of a
`generic' DM, in some sense. The characteristics (and hence utility) of this
representation depend not only on the manner in which it was derived, but
critically on the extent and type of agent heterogeneity present in the finite
population model. While this has lead to many bespoke mean-field models, the
general program is to solve a limiting control problem, and use this to
construct near-optimal decision making for the large (but finite) population.
Thus, by relaxing optimality, a natural problem is on how the mean-field
constrains the search for a policy without loss; what policy/information
structures are sufficient for asymptotic optimality, and how these relate to the
structure of the underlying population. \\[5pt]
\indent For the simplest case of indistinguishable agents (such a population is
sometimes called symmetric) with mean-field type interaction, limiting models
are called mean-field games (MFGs) (introduced contemporaneously in
\cite{Huang_Malhame_Caines_2006, Lasry_Lions_2007}). Previous works have
addressed the existence and uniqueness of mean-field Nash equilibria
\cite{Bardi_Fischer_2019,Carmona_Delarue_Lacker_2016,Light_Weintraub_2022}
and (on the structural front) the existence of symmetric mean-field equilibria
inducing symmetric asymptotic equilibria for the prelimit problem. In the case
of symmetric mean-field teams,
\cite{Sanjari_Saldi_Yüksel_2023,Sanjari_Saldi_Yüksel_2025} further established, without
\textit{a priori} structural restrictions,
the sufficiency of decentralized, mean-field sharing information channels to
compute symmetric optima for a class of continuous, exchangeable cost functions.
Of course, population symmetry is a prohibitive assumption for many
applications. This limitation is particularly evident for network games and
control systems; even when node dynamics are identical, symmetry further
requires that agents be homogeneously coupled in an all-to-all (i.e. complete)
fashion.\\[5pt]
\indent Thus, research efforts have been increasingly directed at relaxing agent
symmetry requirements. Graphon mean-field games (GMFGs)
\cite{Caines_Huang_2019,Caines_Huang_2021,Parise_Ozdaglar_2019} serve as
limiting models when agent interactions (and possibly dynamics) depend on their
nodal location in a finite graph. Mean-field convergence depends on that of the
underlying network topology to a graphon, representing a graph on the node set
$[0,1]$ (see \cite{Lovasz_2012} for an extensive overview of graph limit theory)
which in turn indexes a continuum of graphon-coupled agents (or subpopulation
distributions, in the case of \cite{Caines_Huang_2019,Caines_Huang_2021}) as
total population size and (potentially, but not necessarily) agent-node diversity diverges. For
uncontrolled particle systems involving graphon-type coupling, convergence to a
mean-field has been established via laws of large numbers
\cite{Bayraktar_Chakraborty_Wu_2022,Bet_Coppini_Nardi_2024} and concentration of
measures arguments \cite{Bayraktar_Kim_2024,Bayraktar_Wu_2023} for dense and
sparse regimes, with the latter requiring mild conditions on the initial
distributions and their moments. However, these convergences are for the state
trajectory empirical measures (over the entire population) with respect to the
averaged law of the limiting system. This method of handling unbounded diversity
is a recurring theme in the GMFG literature
\cite{Caines_Huang_2019,Caines_Huang_2021,Cui_Koeppl_2021} and has the
significant disadvantage that it does not preserve subpopulation structures
which may be present in the finite models. Thus, the subpopulation membership of
individual agents is lost in the infinite population, and approximation schemes
depend on sampling (see \cite{Aurell_Carmona_Laurière_2022} for an example in
the linear quadratic case) or else make continuity assumptions on the optimal
strategy {\it a priori} \cite{Caines_Huang_2019}. Thus, owing in part to the
complexity of the limit, structural constraints for GMFG policies remain
outstanding. Additionally, decentralization is typically assumed {\it a priori}
(e.g., see
\cite{Aurell_Carmona_Laurière_2022,Caines_Huang_2019,Caines_Huang_2021,Cui_Koeppl_2021,Lacker_Soret_2022,Vasal_Mishra_Vishwanath_2021,Xu_Gou_Huan_2025})
which may be appropriate for games when the comparison of equilibria is not at
issue, but can be at a great loss when considering global performance, such as
in the case of cooperative games.\\[5pt]
\indent In this paper, we develop structural results for mean-field models using
a different but closely related approach as for GMFGs\footnote{For clarity, our
model corresponds to the case where agents are coupled at the subpopulation
level (as in \cite{Caines_Huang_2019,Caines_Huang_2021}, for instance) and where
the corresponding graphon is both trivial (in the sense that it has an exact
finite graph representation; sometimes called a step-graphon) and constant in
the total population size. Stronger, agent-to-agent coupling schemes have been sometimes distinguished as `Graphon Field Games' (as in \cite{Gao_Tchuendom_Caines_2021}) and are less comparable.}. Specifically, we
consider
discrete time stochastic cooperative games (i.e., teams) played amongst a
population with exactly $M\geq 1$ agent species as the total population
diverges. Each agent of a given species is regarded as belonging to the same
subpopulation, or cluster, but agents still cooperate across clusters, striving
for a global optimum of a shared cost function. Models exhibiting finite
subpopulation structure have appeared sporadically in the literature; in the
case of uncontrolled particle diffusions, \cite{Bayraktar_Chakraborty_Wu_2022}
note a connection between the graphon limit of a sequence of block stochastic
models and a finite system of coupled McKean-Vlasov equations, and
\cite{Bhamidi_Budhiraja_Wu_2019,Budhiraja_Wu_2016} obtain a law of large
numbers/propagation of chaos result for the case of static and time-varying
random graphs. These have also been studied as a generalization of MFGs
admitting finite heterogeneity, and sometimes called multi-population MFGs
(MPMFGs) \cite{Bensoussan_Huang_Laurière_2018,
Feleqi_2013,Fujii_2020,Huang_Malhame_Caines_2006}. However, just as for GMFGs,
these papers make {\it a priori} decentralization assumptions, and do not
address the structure of approximate optima for the finite problem. Most
recently, \cite{Lambrecht_Laurière_2025} considered an almost identical model to
ours, involving a game played amongst finitely-many internally cooperating
subpopulations, but with common disturbance processes at the sub- and complete
population levels. While they were able to prove the existence of
cluster-symmetric closed loop equilibria, this structure was once more imposed
{\it a priori} by lifting the problem to a collection of
cluster-representatives. As such, their analysis is similarly inconclusive for
cooperative games with respect to global optimality.\\[5pt]
\indent Our setup and analyses build directly on those of \cite{Bauerle_2023}
and \cite{Sanjari_Saldi_Yüksel_2025} which each examined the case of symmetric
teams. In the former, an equivalent measure-valued MDP is used to establish the
sufficiency of mean-field sharing information for a centralized team by way of
Blackwell's theorem \cite{Blackwell_1964} whence convergence of the finite
population value functions to that of a mean-field problem is established by way
of the Glivenko-Cantelli lemma. The same measure-valued formulation was used in
\cite{Sanjari_Saldi_Yüksel_2025} to further establish the optimality of
exchangeable policies, for which weak subsequential convergence to an
infinitely-exchangeable limiting strategic measure (along finite marginals) can
be characterized by the convergence in law of random empirical measures.
Subsequently, they obtained two key results: (i) by De Finetti's theorem, the
set of infinitely-exchangeable policies coincides with the set of symmetric
policies which are conditionally independent given a common random variable
(which turns out to be the weak limit of the state empirical distributions
$\as$, called the directing measure) and (ii) that the finite population value
functions converge to that of the mean-field via the aforementioned
measure-valued formulation. Further, the limiting value function corresponds to
a McKean-Vlasov type problem with a single representative agent, which is
optimal for the original problem in light of (i).\\[5pt]
\indent In the case of finitely-many cooperating subpopulations, the main technical
distinction is that we work with the joint convergence of a collection of
cluster empirical distributions to an array of mean-fields which are dependent,
in general. For this we work with a relaxed notion of exchangeability, which we
call cluster-exchangeability, and develop the necessary mathematical analysis
required due to the agent heterogeneity.\\[5pt]
\indent This paper is organized as follows. In Section 2 we introduce the
stochastic dynamic team problem with finitely-many agent species, along with the
main assumptions, definitions and notational devices used throughout the sequel.
In Section 3 we present a measure-valued MDP representation of the original
problem, establish equivalence with the ensemble, and show that
cluster-exchangeable policies are optimal. In Section 4 we turn our attention to
the mean-field limit, developing some theory on the weak convergence of
cluster-exchangeable processes and prove the main result on the structure of
optima for the mean-field limit. Finally, Section 5 introduces a multi-agent
decentralized McKean-Vlasov team problem (with a representative for each agent
cluster) and establishes two important results: first on computing mean-field
solutions via value iterations for the McKean-Vlasov team, and second how these
can induce approximate optima for the finite population team. Over this course
we make the following main contributions:
\begin{enumerate}[1.]
  \item We examine the structure of optimal policies for stochastic teams
  involving finitely many agent species along with their mean-field limits. By
  grouping agents into symmetric clusters, the system admits a form of partial
  exchangeability which we call cluster-exchangeability. For the finite
  ensemble, we establish equivalence with a MDP on the array of cluster
  empirical measures in Theorem \ref{thm5}, and apply this formulation in
  Theorem \ref{thm6} to establishes the optimality of $N$-cluster-exchangeable
  (C-Ex$^N$) Markovian policies.
  \item We provide discussion generalizing classical representation/convergence
  results for fully exchangeable processes to those which are only infinitely
  cluster-exchangeable (C-Ex). Subsequently, Theorem \ref{thm7} extends the
  optimality of cluster-exchangeable policies to the infinite ensemble which
  (via a De Finetti-type representation for C-Ex processes) justifies the
  applicability of a mean-field MDP representation for the infinite ensemble.
  \item We develop the above mean-field MDP, along with a mean-field dynamic
  program and establish a detailed relationship between its policies and those
  of the infinite ensemble in Theorems \ref{thm9}-\ref{thm10}. In particular,
  mean-field policies induce cluster-symmetric, decentralized and conditionally
  independent (under cluster mean-field sharing) policies for the infinite
  population which incur the same cost. 
  \item As a main result, we refine Theorem \ref{thm7} in Theorem \ref{thm11} to
  establish existence of an optimal policy for the infinite ensemble which
  admits the above structural and decentralization characteristics, and is
  realized by the mean-field dynamic program.
  \item In Corollary \ref{cor2} , this same mean-field dynamic program is shown
  to correspond to a decentralized McKean-Vlasov team of cluster-representative
  agents exhibiting a piecewise constant interaction structure. Hence, the
  preceding analysis justifies the optimality of this simplified representation
  for the infinite population without any \textit{a priori} constraints on the
  centralized IS for admissible policies. Additionally, Corollary \ref{cor3}
  demonstrates how mean-field solutions induce decentralized, cluster-symmetric
  policies for the finite population which are asymptotically optimal in the
  large population limit.
  \item Through these results and their supporting analyses, we generalize
  recent results on the case of a single cluster team to the case involving
  finitely many agent types.
\end{enumerate}

\section{Cluster-Exchangeable Team Models}\label{sec1}
In this section, we introduce a $N$-agent dynamic team problem and its infinite
population limit. The essential structural feature of these models is that
agents are symmetric within each of $M\geq 1$ subpopulations, which we refer to
as clusters. We begin by defining our notations, including those needed to
describe the population structure and work with empirical measures over standard
Borel spaces. Then, after presenting the models and their optimality criteria
(to serve as a benchmark for subsequent information structure (IS) relaxations)
we summarize their structural characteristics by introducing the notion of
cluster-exchangeability. From this we derive two policy classes: those which are
cluster-exchangeable, and those which further admit a conditionally independent
and cluster-symmetric structure. Both of these reflect the structure of the
ensemble. In subsequent sections, we examine the optimality of these subclasses
along with their informational demands.

\subsection{Finite Population Teams \texorpdfstring{$\mf{G}^N$}{TEXT}}\label{sec2}
We begin with notation. For $M\geq 1$ henceforth fixed, let $V=\{1,2,\dots,M\}$
index a set of agent clusters comprising a population of size $N\geq M$ (to
preclude empty clusters) indexed by $\mc{N}=\{1,2,\dots,N\}$. Each node $j\in V$
is associated with a subpopulation, or cluster, indexed by
$\mc{C}^N_j\neq\emptyset$ so that $\mc{C}^N:=\{\mc{C}^N_j\}_{j\in V}\subset
2^\mc{N}$ forms a nonempty partition of $\mc{N}$. Each $\DM$ is associated to
its cluster by a function $C^N:\mc{N}\rightarrow V$, where $C^N(i)=j$ if and
only if $i\in\mc{C}^N_j$. The subpopulation size of cluster $j\in V$ is denoted
$N^N_j:=|\mc{C}^N_j|$ when the complete population has size $N\geq M$
\footnote{When considering $N\geq M$ fixed we sometimes omit the $N$-dependence
of these objects, writing $\{N_j\}_{j\in V}$, $\{\mc{C}_j\}_{j\in V}$, and $C$
rather than $\{N^N_j\}_{j\in V}$, $\{\mc{C}_j^N\}_{j\in V}$, $C^N$ and so
on.}.\\[5pt] 
\indent Some notational conventions are needed to compactly/flexibly express
joint quantities over the entire population and individual clusters. With
$\mbb{Y},\{\mbb{Y}^j\}_{j\in V}$ a collection of standard Borel spaces,
$y^i_t,z^i_t\in\mbb{Y}^{C(i)}$ for $i, t\in\mbb{N}$ and $N\geq M$ fixed, the
array of points
$(y^1_t,y^2_t,\dots,y^N_t)\in\mbb{Y}^{C(1)}\times\mbb{Y}^{C(2)}\times\cdots\times\mbb{Y}^{C(N)}$
will be denoted with the shorthand $y^{1:N}_t$. Similarly, if $A\subset\mc{N}$
indexes a collection $\{y^i_t\}_{i\in A}$, then $y^A_t:=(y^i)_{i\in A}$. Tuple
parentheses take precedence over these notations, such that
$(y^{1:N}_t,z^{1:N}_t):=((y^1_t,z^1_t),(y^2_t,z^2_t),\dots,(y^M_t,z^M_t))$. When
additionally indexing time, we set
\[y^{1:N}_{0:t}:=(y^{1:N}_0,y^{1:N}_1,\dots,y^{1:N}_t)\quad\text{and}\quad(y^{1:N}_{0:t},z^{1:N}_{0:t})=((y^{1:N}_0,z^{1:N}_0),(y^{1:N}_1,z^{1:N}_1),\dots,(y^{1:N}_t,z^{1:N}_t))\]
for $t\geq 0$. We will denote the Cartesian product
$\mbb{Y}^{1:M}:=\prod_{j=1}^M\mbb{Y}^j$, with an analogous definition for
$\mbb{Y}^{1:\infty}$ in case we have infinitely-many standard Borel spaces.
These are always equipped with the product topology and corresponding Borel
$\sigma$-algebra.\\[5pt]
\indent The set of $N$-point empirical measures on $\mbb{Y}$ will be given by
\begin{align}
  \mc{P}^N_E(\mbb{Y}):=\left\{\nu\in\mc{P}(\mbb{Y})\bigsetsep\nu=\frac{1}{N}\sum_{i=1}^N\delta_{y^i},\;y^i\in\mbb{Y}\;\text{for}\;1\leq i\leq N\right\},\label{eq9}
\end{align}
such that any empirical measure over $\mbb{Y}$ resides within
$\mc{P}_E(\mbb{Y}):=\cup_{N\geq 1}\mc{P}^n_E(\mbb{Y})$. Accordingly, we
introduce a function $\mu[\cdot]$ which, for a point
$y^{1:N}\in\prod_{i=1}^N\mbb{Y}$, instantiates its empirical distribution via
$\mu[y^{1:N}]:=\frac{1}{N}\sum_{i=1}^N\delta_{y^i}$. This notation will be used
to construct empirical measures over a variety of spaces and with arrays of any
(finite) number of elements (i.e., without a more precise notation such as
$\mu_{\mbb{Y}}^N[\cdot]:\prod_{i=1}^N\mbb{Y}\rightarrow\mc{P}^N_E(\mbb{Y})$).
Finally, we write $\mc{P}(\mbb{Y}^2|\mbb{Y}^1)$ for the set of stochastic
kernels (i.e., measurable maps from $\mbb{Y}^1$ to $\mc{P}(\mbb{Y}^2)$) and, if
$Y$ is a $\mbb{Y}$-valued random variable, we write $\mc{L}(Y)$ to denote its
law.\\[5pt]
\indent Proceeding to the finite population team problem, each agent
$i\in\mc{N}$ is equipped with standard Borel state, action and disturbance
spaces $\mbb{X}^{C(i)}$, $\mbb{U}^{C(i)}$ and $\mbb{W}^{C(i)}$ respectively,
identical within clusters, and dynamics given by
$f_{C(i)}:\mbb{X}^{C(i)}\times\mbb{U}^{C(i)}\times\mc{P}(\mbb{X}^1)\times\mc{P}(\mbb{X}^2)\times\cdots\mc{P}(\mbb{X}^M)\times\mbb{W}^{C(i)}\rightarrow\mbb{X}^{C(i)}$
such that
\begin{align}
  x^i_{t+1}=f_{C(i)}(x^i_t,u^i_t,\mu^{1:M}_t,w^i_t),\quad\text{and}\quad x^{1:N}_{t+1}=F^N(x^{1:N}_t,u^{1:N}_t,\mu^{1:M}_t,w^{1:N}_t),\quad \forall\:t\geq 0\label{eq1}
\end{align}
where $F^N=(f_{C(1)},f_{C(2)},\dots,f_{C(N)})$ denotes the joint team dynamics
on the common state $\bar{\mbb{X}}^N:=\prod_{i\in\mc{N}}\mbb{X}^{C(i)}$ and
action $\bar{\mbb{U}}^N:=\prod_{i\in\mc{N}}\mbb{U}^{C(i)}$ spaces. Note that
agents are only weakly coupled via the cluster empirical measures
$\mu^j_t:=\mu[x^{\mc{C}_j}_t]\in\mc{P}_E(\mbb{X}^j)$, $j\in V$. Initial states
and disturbance processes are $\iid$ within clusters, but only independent
between. That is,
\begin{align}
  x^i_0\overset{\iid}{\sim}\nu^{j}_0\in\mc{P}(\mbb{X}^j),\quad\text{and}\quad w^i_t\overset{\iid}{\sim}\nu^{j}\in\mc{P}(\mbb{W}^j)\quad \forall\:i\in\mc{C}_j,\,t\geq 0,\,j\in V,
\end{align}
whence we set $\nu_0:=\bigotimes_{i=1}^N\nu^{C(i)}_0$ and
$\nu:=\bigotimes_{i=1}^N\nu^{C(i)}$ (with the population size $N\geq M$ left
implicit) to be the joint distributions for the initial states and disturbances,
respectively. The dynamics (\ref{eq1}) can be equivalently realized by a
collection of stochastic kernels
$\mc{T}_j\in\mc{P}(\mbb{X}^j|\mbb{X}^j\times\mbb{U}^j\times\mc{P}(\mbb{X}^1)\times\mc{P}(\mbb{X}^2)\times\cdots\times\mc{P}(\mbb{X}^M))$
for $j\in V$. By an abuse of notation, we take $f_i$, $\mc{T}_i$, $\mbb{X}^i$,
$\mbb{U}^i$ and $\mbb{W}^i$ as shorthand for $f_{C(i)}$, $\mc{T}_{C(i)}$,
$\mbb{X}^{C(i)}$, $\mbb{U}^{C(i)}$ and $\mbb{W}^{C(i)}$ (respectively) when
indexing with $i\in\mc{N}$. To make the dependence of $\mu^{1:M}_t$ on a
particular joint state $x^{1:N}\in\bar{\mbb{X}}^N$ explicit, define
$\mu^j[x^{1:N}]:=\mu[x^{\mc{C}_j}]$ for the empirical measure over a single
cluster, and
$\mu^{1:M}[x^{1:N}_t]:=(\mu^1[x^{1:N}_t],\mu^2[x^{1:N}_t],\dots,\mu^{M}[x^{1:N}_t])$
the array of cluster-empirical measures corresponding to
$x^{1:N}_t\in\bar{\mbb{X}}^N$. Otherwise, the relation
$\mu^j_t=\mu[x^{\mc{C}_j}_t]$ for $j\in V$ is left implicit. Rather than total
symmetry, agents here are only symmetric within each cluster, which is manifest
in the form of the joint transition kernel
$\mc{T}^N\in\mc{P}(\bar{\mbb{X}}^N|\bar{\mbb{X}}^N\times\bar{\mbb{U}}^N)$;
$\forall\:A\in\mc{B}(\bar{\mbb{X}}^N)$,
\begin{align}
  \mc{T}^N(A|x^{1:N}_t,u^{1:N}_t):=P(x^{1:N}_{t+1}\in A|x^{1:N}_t,u^{1:N}_t)=\prod_{j\in V}\prod_{i\in\mc{C}_j}\mc{T}_{j}(x^i_{t+1}\in A|x^i_t,u^i_t,\mu^{1:M}_t).\label{eq3}
\end{align}
\indent Since one of our primary objectives is to study global optimality under
various information structure (IS) relaxations, agents are \textit{a priori}
endowed with completely centralized information given by
\begin{align}
  \mc{I}^{i}_t:=\{x^{1:N}_{0:t},u^{1:N}_{0:t-1}\},\quad\forall\:i\in\mc{N}.
\end{align}
With this serving as a benchmark, we consider performance under the (i)
centralized cluster mean-field sharing (cCMF) and (ii) decentralized cluster
mean-field sharing (dCMF) IS, which are generated by the processes
\begin{align}
  \quad\mc{I}^{\text{cCMF},i}_t:=\{\mu^{1:M}_{0:t}\},\quad\text{and}\quad \mc{I}^{\text{dCMF},i}_{t}:=\{x^i_{0:t},u^i_{0:t-1},\tilde{\mu}^{1:M}_{0:t}\}\quad\forall\:i\in\mc{N},\,t\geq 0,\label{eq4}
\end{align}
respectively. Note that the process $\tilde{\mu}^{1:M}_{0:t}$ used in the dCMF
IS is not the array of state empirical measures, but rather the mean-field
process (developed in $\S$\ref{sec10}). The main utility of the cCMF IS is to facilitate a
centralized optimization for an equivalent finite population
representation.\\[5pt]
\indent Corresponding to the completely centralized IS, define the history spaces
$\mbb{H}^N_0:=\bar{\mbb{X}}^N$ and
$\mbb{H}^N_t=\mbb{H}^N_{t-1}\times(\bar{\mbb{X}}^N\times\bar{\mbb{U}}^N)$ for
$t\geq 1$. Then an admissible, centralized policy on the team's joint action at
time $t\geq 0$ is just a stochastic kernel
$\pi_t\in\bar{\Pi}^{N}_t:=\mc{P}(\bar{\mbb{U}}^N|\mbb{H}^N_t)$. Policies over
the finite and infinite horizons are given by
$\bar{\Pi}^{N,T}:=\prod_{t=0}^{T-1}\bar{\Pi}^N_t$ and $\bar{\Pi}^N:=\prod_{t\geq
0}\bar{\Pi}^N_t$, respectively. The set of deterministic policies (generated by
measurable maps $\gamma_t:\mbb{H}^N_t\rightarrow\bar{\mbb{U}}^N$ for $t\geq 0$)
is denoted by replacing `$\Pi$' with `$\Gamma$' in the above notation, whence we
recognize $\bar{\Gamma}^N_t\subset\bar{\Pi}^N_t$ by considering
$\delta_{\gamma_t}\in\mc{P}(\bar{\mbb{U}}^N|\mbb{H}^N_t)$.\\[5pt]
\indent To complete the team description, we propose a single cost function
which the ensemble strives to collectively infimize. This is given by one of
\begin{align*}
  &J^N_T(\pi, \nu_0):=\E^{\pi}_{\nu_0}\bigg[\sum_{t=0}^{T-1}\beta^tc^N(x^{1:N}_t,u^{1:N}_t)\bigg],\quad\forall\:\pi\in\bar{\Pi}^{N,T}\\
  &J^N(\pi,\nu_0):=\E^{\pi}_{\nu_0}\bigg[\sum_{t=0}^\infty\beta^tc^N(x^{1:N}_t,u^{1:N}_t)\bigg],\quad\forall\:\pi\in\bar{\Pi}^N
\end{align*}
corresponding to the finite and infinite horizons, respectively. In particular,
the stagewise cost
$c^N:\bar{\mbb{X}}^N\times\bar{\mbb{U}}^N\rightarrow\mbb{R}_+$ is taken to be of
the form
\begin{align}
  c^N(x^{1:N}_t,u^{1:N}_t)=\sum_{j\in V}\frac{1}{N_j}\sum_{i\in\mc{C}_j}c^j(x^i_t,u^i_t,\mu^{1:M}_t)\label{eq7}
\end{align}
for a collection of component cost functions
$c^j:\mbb{X}^j\times\mbb{U}^j\times\mc{P}(\mbb{X}^1)\times\mc{P}(\mbb{X}^2)\times\cdots\times\mc{P}(\mbb{X}^M)\rightarrow\mbb{R}_+$,
$j\in V$. That is, each agent's contribution to the collective running cost
depends on their cluster membership, and is weighted according to the size of
their cluster. Moreover, it can be expressed as a comparison of their individual
state-action profile with the empirical activity of each subpopulation and/or a
comparison of the empirical distributions themselves. A policy $\pi^{\ast}$ is
team-optimal for these problems if and only if
\begin{align*}
  &J_T^{N,\ast}(\nu_0):=J_T^N(\pi^{\ast},\nu_0)=\inf_{\pi\in\bar{\Pi}^{N,T}}J^N_T(\pi,\nu_0),\\
  &\qquad\text{or}\quad J^{N,\ast}(\nu_0):=J^N(\pi^{\ast},\nu_0)=\inf_{\pi\in\bar{\Pi}^{N}}J^N(\pi,\nu_0)
\end{align*}
with the provisio that $\pi^{\ast}\in\bar{\Pi}^{N,T}$ in the former and
$\pi^{\ast}\in\bar{\Pi}^{N}$ in the latter. We shall refer to the finite and
infinite horizon discounted, centralized team problems with $\mf{G}^N_T$, and
$\mf{G}^N_\infty$, respectively, or to either interchangeably with just
$\mf{G}^N$.

\subsection{Infinite Population Teams \texorpdfstring{$\mf{G}^\infty$}{TEXT} and Cluster-Exchangeability}\label{sec3}
The infinite population problem is constructed in analogy with $\mf{G}^N$, but
now with $N=\infty$, or $\mc{N}=\mbb{N}$. The state and action spaces are now
$\bar{\mbb{X}}:=\prod_{i\in\mbb{N}}\mbb{X}^{C(i)}$,
$\bar{\mbb{U}}:=\prod_{i\in\mbb{N}}\mbb{U}^{C(i)}$, with corresponding history
spaces $\mbb{H}_0:=\bar{\mbb{X}}$,
$\mbb{H}_t:=\mbb{H}_{t-1}\times(\bar{\mbb{X}}\times\bar{\mbb{U}})$  for $t\geq
1$. Admissible policies are those residing in
$\bar{\Pi}^T:=\prod_{t=0}^{T-1}\bar{\Pi}_t$, or $\bar{\Pi}:=\prod_{t\geq
0}\bar{\Pi}_t$ where $\bar{\Pi}_t:=\mc{P}(\bar{\mbb{U}}|\mbb{H}_t)$. Policies
are assessed according to the new cost criteria
\begin{align}
  &J_T(\pi,\nu_0):=\limsup_{N\rightarrow\infty}J^N_T(\pi|_N,\nu_0)\label{eq8}\\
  &J(\pi,\nu_0):=\limsup_{N\rightarrow\infty}J^N(\pi|_N,\nu_0)
\end{align}
where $\pi\in\bar{\Pi}^{T}$ in the former and $\pi\in\bar{\Pi}$ in the latter.
The notation $\pi|_N$ denotes the policy obtained by restricting the joint
measure on $\prod_{t=0}^{T-1}\bar{\mbb{X}}\times\bar{\mbb{U}}$ induced by $\pi$
to $\prod_{t=0}^{T-1}\bar{\mbb{X}}^N\times\bar{\mbb{U}}^N$ (and similarly for
the infinite horizon). In particular, if
$P^\pi_{\nu_0}\in\mc{P}(\prod_{t=0}^{T-1}\bar{\mbb{X}}\times\bar{\mbb{U}})$ is
the strategic measure induced by $\pi\in\bar{\Pi}^T$ and $\nu_0$, then
(\ref{eq8}) can be written
\begin{align}
  J_T(\pi,\nu_0)=\limsup_{N\rightarrow\infty}\int_{\prod_{t=0}^{T-1}\bar{\mbb{X}}^N\times\bar{\mbb{U}}^N}\sum_{t=0}^{T-1}\beta^tc^NdP^\pi_{\nu_0}.
\end{align}
A policy $\pi^{\ast}$ is team-optimal for these if and only if
\begin{align*}
  J^\ast_T(\nu_0):=J_T(\pi^{\ast},\nu_0)=\inf_{\pi\in\bar{\Pi}^{T}}J_T(\pi,\nu_0),\quad\text{or}\quad J^\ast(\nu_0):=J(\pi^{\ast},\nu_0)=\inf_{\pi\in\bar{\Pi}}J(\pi,\nu_0)
\end{align*}
with the provisio that $\pi^{\ast}\in\bar{\Pi}^{T}$ in the former and
$\pi^{\ast}\in\bar{\Pi}$ in the latter. We shall refer to the infinite
population, centralized team problems with $\mf{G}^\infty_T$ and
$\mf{G}^\infty_\infty$, respectively, or to either of them interchangeably with
just $\mf{G}^\infty$.\\[5pt]
\indent Regardless of the population size, all analyses in the sequel depend on
the following assumptions.
\begin{assumption}\label{ass1} For $M\leq N\leq \infty$, the following
  statements regarding both of $\mf{G}^N_T$, $\mf{G}^N_\infty$ are true:
  \begin{enumerate}[(i)]
    \item $\forall j\in V$, the spaces $\mbb{X}^j$ and $\mbb{U}^j$ are compact.
    \item $\forall j\in V$, $w\in\mbb{W}$, the functions
    $f^j(\cdot,\cdot,\cdot,w):\mbb{X}^j\times\mbb{U}^j\times\mc{P}(\mbb{X}^1)\times\mc{P}(\mbb{X}^2)\times\cdots\times\mc{P}(\mbb{X}^M)\rightarrow\mbb{X}^j$
    are jointly continuous.
    \item $\forall j\in V$, the stagewise cost components
    $c^j:\mbb{X}^j\times\mbb{U}^j\times\mc{P}(\mbb{X}^1)\times\mc{P}(\mbb{X}^2)\times\cdots\times\mc{P}(\mbb{X}^M)\rightarrow\mbb{R}_+$
    are jointly continuous and bounded.
    \item $\forall j\in V$, $\lim_{N\rightarrow\infty}N_j^N/N>0$.
  \end{enumerate}
\end{assumption}
Under assumption \ref{ass1} (i-iii), (and by viewing the centralized team
equivalently as a single agent with augmented state/action spaces) classical
measurable selection conditions hold and an optimal deterministic, Markovian
policy for $\mf{G}^N_T$ is available via the Bellman recursions:
\begin{align}
  &J^N_{T-1,T}(x^{1:N}):=\inf_{u^{1:N}\in\bar{\mbb{U}}^N}\big\{c(x^{1:N},u^{1:N})\big\}\label{eq95}\\
  &J^N_{t,T}(x^{1:N}):=\inf_{u^{1:N}\in\bar{\mbb{U}}^N}\bigg\{c(x^{1:N},u^{1:N})+\beta\int_{\bar{\mbb{X}}^N}J^N_{t+1,T}(x)\mc{T}^N(dx|x^{1:N},u^{1:N})\bigg\}\label{eq96}
\end{align}
where $0\leq t\leq T-2$, $x^{1:N}\in\ess$. Then, by the monotone convergence
theorem and standard contraction argument for the Bellman operator realizing
these recursions, the limit
\begin{align}
  J^N_\infty(x^{1:N}):=\lim_{T\rightarrow\infty}J^N_{t,T}(x^{1:N}),\quad \forall\:x^{1:N}\in\bar{\mbb{X}}^N\label{eq97}
\end{align}
exists, is insensitive to the value of $t\geq 0$ and thus satisfies the
discounted cost optimality equation
\begin{align}
  J^N_\infty(x^{1:N})=\inf_{u^{1:N}\in\eas}\bigg\{c(x^{1:N},u^{1:N})+\beta\int_{\ess}J^N_\infty(x)\mc{T}^N(dx|x^{1:N},u^{1:N})\bigg\}\label{eq98}
\end{align}
which (again under assumption \ref{ass1}) admits a measurable selector inducing
an optimal deterministic stationary policy for $\mf{G}^N_\infty$ (see \cite[$\S$
4.2]{Hernandez-Lerma_Lasserre_1996}). However, standard dynamic programming
alone is insufficient to obtain the desired relaxations on the IS or the
conditional independence of agent decisions. Rather, this will critically depend
on the on the various forms of symmetry found in $\mf{G}^N$ which can be
summarized as a relaxed form of exchangeability; an invariance under a subgroup
of the $\mc{N}$-permutations which we shall call cluster permutations.
\begin{definition}[Cluster-Exchangeability]\label{def1} For any
  $A\subset\mbb{N}$, let $\mc{S}_A$ denote the set of permutations of $A$. Then
  $\mc{S}_{\mc{C}^N}$ is called the set of cluster-permutations (with respect to
  partition $\mc{C}^N$) and is characterized by
    \[\sigma\in\mc{S}_{\mc{C}^N}\quad\Longleftrightarrow\quad\sigma\big|_{\mc{C}^N_j}\in\mc{S}_{\mc{C}^N_j},\quad\forall\:j\in
    V.\] A random variable $Y^{1:N}$ is then called $N$-cluster-exchangeable
    (C-Ex$^N$) with respect to $\mc{C}^N$ if and only if
    $\mc{L}(Y^{1:N})=\mc{L}(Y^{\sigma(1):\sigma(N)})$ for all
    $\sigma\in\mc{S}_{\mc{C}^N}$. An infinite collection $Y^{1:\infty}$ is
    instead called infinitely cluster-exchangeable (C-Ex) if and only if
    $\mc{L}(Y^{1:N})=\mc{L}(Y^{\sigma(1):\sigma(N)})$ for all
    $\sigma\in\mc{S}_{\mc{C}^N}$ and each $N\geq M$.
\end{definition}
With respect to the broader literature, cluster-exchangeability is nothing more
than a form of partial exchangeability suited to the subpopulation structure of
our model. It is manifest by the symmetries of our team problem as follows: for
all $x^{1:N}\in\ess$, $u^{1:N}\in\eas$, we have the invariances:
\begin{align*}
  \begin{array}{l}
    \mu^{1:M}[x^{1:N}]=\mu^{1:M}[x^{\sigma(1):\sigma(N)}],\\
    c^N(x^{1:N},u^{1:N})=c^N(x^{\sigma(1):\sigma(N)},u^{\sigma(1):\sigma(N)}),\\
    \mc{T}^N(d\bar{x}^{1:N}|x^{1:N},u^{1:N})=\mc{T}^N(d\bar{x}^{\sigma(1):\sigma(N)}|x^{\sigma(1):\sigma(N)},u^{\sigma(1):\sigma(N)})
  \end{array}
\end{align*}
$\forall\,\sigma\in\mc{S}_{\mc{C}^N}$, which is easily seen by reviewing their
definitions. That is, the empirical measure function $\mu^{1:M}[\cdot]$,
stagewise cost $c^N$ and transition kernel $\mc{T}^N$ defining $\mf{G}^N$ are
each C-Ex$^N$. From this observation, a natural question is on the performance
of policies reflecting this symmetry -- defined by constraining the dependence
structure of agents' actions.
\begin{definition}\label{def5} For $\mf{G}^N_T$, we define the set of C-Ex$^N$
  policies $\Pi^{N,T}_{\cex}\subset\bar{\Pi}^{N,T}$ as
  \begin{align}
    \Pi^{N,T}_{\cex}:=\bigg\{\pi_{0:T-1}\in&\bar{\Pi}^{N,T}\setsep\pi_t(du^{1:N}_t|x^{1:N}_{0:t},u^{1:N}_{0:t-1},\mu^{1:M}_t)=\notag\\
    &\pi_t(du^{\sigma(1):\sigma(N)}_t|x^{\sigma(1):\sigma(N)}_{0:t},u^{\sigma(1):\sigma(N)}_{0:t-1},\mu^{1:M}_t),\;\forall\:\sigma\in\mc{S}_{\mc{C}^N},\,0\leq t<T\bigg\}\label{eq67}.
  \end{align}
   Subsequently, the class of $N$-cluster-symmetric (C-Sym$^N$) and
   conditionally independent policies $\Pi^{N,T}_{\csym}\subset\Pi^{N,T}_{\cex}$
   is given by
  \begin{align}
    \Pi^{N,T}_{\csym}:=\bigg\{\pi_{0:T-1}\in&\bar{\Pi}^{N,T}\setsep\pi_t(du^{1:N}_t|x^{1:N}_{0:t},u^{1:N}_{0:t-1},\mu^{1:M}_t)=\notag\\
    &\prod_{j\in V}\prod_{i\in\mc{C}^N_j}\pi^j_t(du^i_t|x^{1:N}_t,u^{1:N}_t,\mu^{1:M}_t),\;0\leq t<T\bigg\}\label{eq68}
  \end{align}
  where each $\pi^j_t\in\mc{P}(\mbb{U}^j|\mbb{H}^N_t)$ in (\ref{eq68}) must
  itself be C-Ex$^N$. For $\mf{G}^N_\infty$ we write $\Pi^N_{\cex}$,
  $\Pi^N_{\csym}$, and for $\mf{G}^\infty$ we write $\Pi^T_{\cex}$,
  $\Pi^T_{\csym}$, $\Pi_{\cex}$, $\Pi_{\csym}$, all of which follow analogous
  definitions.
\end{definition}
Under cluster-exchangeable policies, all within-clusters reconfigurations of a
given joint action are equiprobable, and the controller is indifferent to which
agent (up to their cluster-membership) executes each component. In contrast,
cluster-symmetric policies impose further structure by requiring agents select
their actions independently, but under the same law within clusters.

\section{Optimality of C-Ex\texorpdfstring{$^N$}{TEXT} Policies for \texorpdfstring{$\mf{G}^N$}{TEXT}}\label{sec4}
To establish the optimality of $\Pi^{N,T}_{\cex}$ and $\Pi^{N}_{\cex}$ for
$\mf{G}^N$, we follow the example of \cite{Bauerle_2023} for symmetric ($M=1$)
teams to develop an equivalent measure-valued formulation of $\mf{G}^N$. The
main purpose of this reformulation is to facilitate optimization under the cCMF
IS and establish the optimality of C-Ex$^N$ policies for $\mf{G}^N$. Rather than
representing the entire joint state up to a $\mc{N}$-permutation using its
empirical measure, our (putative) MDP tracks an array of $M$ distinct empirical
measures $\{\mu^{1:M}_{t}\}_{t\geq 0}$, preserving the subpopulation structure
of $\mf{G}^N$. Optimality of C-Ex$^N$ policies then follows the invariance of
the empirical measures under cluster-permutations, along with their sufficiency
for computing the running costs/value function.

\subsection{Equivalent Empirical Measure MDP
\texorpdfstring{$\wh{\mf{G}}^N$}{TEXT}}\label{sec5} Fixing $N\geq M$, our
objective is to form a proper MDP from the
$\mc{P}^{N_1}_E(\mbb{X}^1)\times\mc{P}^{N_2}_E(\mbb{X}^2)\times\cdots\times\mc{P}^{N_M}_E(\mbb{X}^M)$-valued
process $\{\mu^{1:M}_t\}_{t\geq 0}$, denoting this state space $\wh{\mbb{X}}^N$.
The ambient action space is given by
$\wh{\mbb{U}}^N:=\mc{P}^{N_1}_E(\mbb{X}^1\times\mbb{U}^1)\times\mc{P}^{N_2}_E(\mbb{X}^2\times\mbb{U}^2)\times\cdots\times\mc{P}^{N_M}_E(\mbb{X}^M\times\mbb{U}^M)$,
with the constraint that for any state $\mu^{1:M}\in\wh{\mbb{X}}^N$, actions
must be selected from the subset
\begin{align}
  \wh{\mbb{U}}^N(\mu^{1:M}):=\left\{\mu^{1:M}[(x^{1:N},u^{1:N})]\in\wh{\mbb{U}}^N\;\big|\;u^{1:N}\in\eas,\;x^{1:N}\in\ess\;\;\text{and}\;\;\mu^{1:M}[x^{1:N}]=\mu^{1:M}\right\}\label{eq10}
\end{align}
such that the joint state-action empirical measure in entry $j$ of any
$\theta^{1:M}\in\wh{\mbb{U}}^N(\mu^{1:M})$ has marginals on $\mbb{X}^j$ in
agreement with $\mu^j$. When dealing with $N\geq M$ fixed, we will simply write
$\wh{\mbb{X}}$, $\wh{\mbb{U}}$,
$\{\wh{\mbb{U}}(\mu^{1:M})\setsep\mu^{1:M}\in\wh{\mbb{X}}\}$. Henceforth, any
such product space is viewed as a subset of the joint probability measures by
defining (in the case of $\wh{\mbb{X}}$)
\begin{align}
  \mu^{1:M}(A):=\int_A\prod_{j\in V}\mu^j(dx^j),\quad\forall\:A\in\mc{B}(\mbb{X}^{1:M}).\label{eq129}
\end{align}
Further, note that endowing $\wh{\mbb{X}}$ and $\wh{\mbb{U}}$ with the topology
of weak convergence renders the empirical measure function $\mu^{1:M}[\cdot]$
weakly continuous irrespective of the subjacent space.\\[5pt]
\indent Admissible policies for the measure-valued process are those allowed to
depend on the (measure-valued) state-action profile given by
\begin{align}
  \wh{\mc{I}}_t:=\{\mu^{1:M}_{0:t},\theta^{1:M}_{0:t-1}\},\quad t\geq 0\label{eq16}
\end{align}
where each $\theta^{1:M}_t\in\wh{\mbb{U}}$, and whence we may define policies in
the same manner as the original team problem. For history spaces
$\wh{\mbb{H}}_0:=\wh{\mbb{X}}$,
$\wh{\mbb{H}}_t:=\wh{\mbb{H}}_{t-1}\times(\wh{\mbb{X}}\times\wh{\mbb{U}})$ for
$t\geq 1$, an admissible policy at time $t\geq 0$ is a stochastic kernel
$\hat{\pi}_t\in\mc{P}(\wh{\mbb{U}}|\wh{\mbb{H}}_t)$ complying with the action
constraint in (\ref{eq10}) such that
$\hat{\pi}_t(\wh{\mbb{U}}(\mu^{1:M})|\mu^{1:M})=1$ for all
$\mu^{1:M}\in\wh{\mbb{X}}$. The set of these shall be denoted $\wh{\Pi}^N_t$.
Policies over a finite horizon $T\geq 1$ are then taken from the Cartesian
product $\wh{\Pi}^{N,T}:=\prod_{t=0}^{T-1}\wh{\Pi}^N_t$, and over an infinite
horizon from $\wh{\Pi}^N:=\prod_{t\geq 0}\wh{\Pi}^N_t$. Just as for the ensemble
team problem, sets of deterministic policies are denoted by replacing
`$\wh{\Pi}$' with `$\wh{\Gamma}$' in the preceding notations.\\[5pt]
\indent Our first result is analogous to \cite[lemma 3.1]{Bauerle_2023}, and
establishes that the local empirical measures
$\mu^{1:M}[x^{1:N}]\in\wh{\mbb{X}}$ with a resulting action
$\theta^{1:M}\in\wh{\mbb{U}}(\mu^{1:M}[x^{1:N}])$ can form an equivalent
representation of $x^{1:N}\in\ess$ and $u^{1:N}\in\eas$ for the purpose of
computing the stagewise costs for $\mf{G}^N$.
\begin{lemma}\label{lem1} Let $x^{1:N}\in\ess$ and $u^{1:N}\in\eas$. Then for
the state $\mu^{1:M}=\mu^{1:M}[x^{1:N}]$, there exists an action
$\theta^{1:M}\in\wh{\mbb{U}}(\mu^{1:M}[x^{1:N}])$ such that
  \begin{align*}
    c^N(x^{1:N},u^{1:N})=\sum_{j\in V}\int_{\mbb{X}^j\times\mbb{U}^j}c^j(x,u,\mu^{1:M})\theta^j(dx\times du)=:\wh{c}(\mu^{1:M},\theta^{1:M}).
  \end{align*}
  Conversely, if $\theta^{1:M}\in\wh{\mbb{U}}(\mu^{1:M}[x^{1:N}])$ for some
  $x^{1:N}$, then $\exists\, u^{1:N}\in\eas$ such that
  $c^N(x^{1:N},u^{1:N})=\wh{c}(\mu^{1:M}[x^{1:N}],\theta^{1:M})$.
\end{lemma}
\noindent{\bf Proof}. Fixing $x^{1:N}\in\ess$ and $u^{1:N}\in\eas$, simply take
\begin{align}
  \theta^{1:M}=\mu^{1:M}[(x^{1:N},u^{1:N})]\label{eq13}
\end{align}
such that $\theta^{1:M}\in\wh{\mbb{U}}(\mu^{1:M}[x^{1:N}])$ by definition. Then
\begin{align*}
  c^N(x^{1:N},u^{1:N})&=\sum_{j\in V}\frac{1}{N_j}\sum_{i\in\mc{C}_j}c^j(x^i,u^i,\mu^{1:M}[x^{1:N}])\\
  &=\sum_{j\in V}\int_{\mbb{X}^j\times\mbb{U}^j}c^j(x,u,\mu^{1:M})\theta^j(dx\times du)
  =\wh{c}(\mu^{1:M},\theta^{1:M})
\end{align*}
as claimed. For the converse, take
$\theta^{1:M}\in\wh{\mbb{U}}(\mu^{1:M}[x^{1:N}])$ for some $x^{1:N}\in\ess$
fixed, and observe that by (\ref{eq10}), $\theta^{1:M}$ is of the form
(\ref{eq13}). Then we repeat the same computation as above in
reverse.\hfill{$\qed$}\\[5pt]
\indent Lemma \ref{lem1} also furnishes a stagewise cost function
$\wh{c}:\mc{P}(\mbb{X}^{1:M})\times\mc{P}(\mbb{X}^{1:M}\times\mbb{U}^{1:M})\rightarrow\mbb{R}_+$
for our putative MDP which is consistent with $\mf{G}^N$ and does not depend on
the population size $N\geq M$. In order to ensure dynamical consistency with the
underlying ensemble, the transitions will be expressed as a lifting of the
original dynamics: if $\mu^{1:M}_t=\mu^{1:M}[x^{1:N}_t]$ and
$\theta^{1:M}_t\in\wh{\mbb{U}}(\mu^{1:M}_t)$ corresponds to a joint action
$u^{1:N}_t\in\eas$ for some $t\geq 0$ in the ensemble (conceptually understood
as `ground truth') then we define
$\wh{F}^N:\wh{\mbb{X}}\times\wh{\mbb{U}}\times\mbb{W}^{1:N}\rightarrow\wh{\mbb{X}}$
so that
\begin{align}
  \mu^{1:M}_{t+1}=\wh{F}^N(\mu^{1:M}_t,\theta^{1:M}_t,w^{1:N}_t):=\mu^{1:M}[F^N(\hat{x}^{1:N}_t,\hat{u}^{1:N}_t,\mu^{1:M}_t,w^{1:N}_t)],\quad t\geq 0\label{eq15}
\end{align}
where $(\hat{x}^{1:N}_t,\hat{u}^{1:N}_t)$ represents the true ensemble
$(x^{1:N}_t,u^{1:N}_t)$ up to empirical equivalence within each cluster -- that
is, $\exists\,\sigma\in\mc{S}_{\mc{C}^N}$ so that
$(\hat{x}^{\sigma(1):\sigma(N)}_t,
\hat{u}^{\sigma(1):\sigma(N)}_t)=(x^{1:N}_t,u^{1:N}_t)$. That such a
representative can be found given $\theta^{1:M}_t$ corresponding to
$(x^{1:N}_t,u^{1:N}_t)$ is evident from its definition; it tells us precisely
which state-action pairs reside within each cluster, but not the identity of the
agent to which each belongs. In fact, it is not difficult to see that equality
up to empirical distribution and equality up to a cluster permutation induce the
same equivalence relation on $\ess\times\eas$, which we refer to as empirical
equivalence. Dynamics (\ref{eq15}) then give rise to the time-invariant
Markovian kernel
\begin{align}
  &\wh{\mc{T}}^N(d\mu^{1:M}_{t+1}|\mu^{1:M}_t,\theta^{1:M}_t)=\prod_{j=1}^M\wh{\mc{T}}_j^N(d\mu^j_{t+1}|\mu^{1:M}_t,\theta^j_t).\label{eq134}
\end{align}
Importantly, the recursions in (\ref{eq15}) (along with a suitable
initialization -- see below) induce a distributional consistency between the
measure-valued and ensemble MDPs. To illustrate, let $A\in\mc{B}(\wh{\mbb{X}})$,
$B:=(\mu^{1:M})^{-1}(A)$, $\mu^{1:N}_t\in\wh{\mbb{X}}$ and
$\theta^{1:M}_t\in\wh{\mbb{U}}(\mu^{1:M}_t)$ correspond to some
$(x^{1:N}_t,u^{1:N}_t)\in\ess\times\eas$. Then,
\begin{align*}
  \wh{\mc{T}}^N(A|\mu^{1:M}_t,\theta^{1:M}_t)&=P(x^{\sigma(1):\sigma(N)}_{t+1}\in B|x^{\sigma(1):\sigma(N)}_t,u^{\sigma(1):\sigma(N)}_t, \mu^{1:M}_t)\tag{for some $\sigma\in\mc{S}_\mc{G}$}\\
  &=\int_{B}\prod_{j\in V}\prod_{i\in\mc{C}_j}\mc{T}_j(dx^{\sigma(i)}|x^{\sigma(i)}_t,u^{\sigma(i)}_t,\mu^{1:M}_t)\\
  &=\int_{B}\prod_{j\in V}\prod_{i\in\mc{C}_j}\mc{T}_j(dx^i|x^{i}_t,u^{i}_t,\mu^{1:M}_t)\\
  &=\mc{T}^N(B|x^{1:N}_t,u^{1:N}_t, \mu^{1:M}_t)
\end{align*}
where, in particular, no matter which representative
$(\hat{x}^{1:N}_t,\hat{u}^{1:N}_t)$ we glean from
$(\mu^{1:M}_t,\theta^{1:M}_t)$, $\hat{x}^{1:N}_{t+1}$ has the same conditional
distribution as $x^{1:N}_{t+1}$ (as do $\mu^{1:M}_{t+1}$ and
$\mu^{1:M}[x^{1:N}_{t+1}]$) and consequently the transition to $\mu^{1:M}_{t+1}$
is insensitive to this choice. \\[5pt]
\indent Given the discussion thus far, the finite and infinite horizon
discounted cost criteria for the measure-valued problem are
\begin{align*}
  &\wh{J}_T(\hat{\pi},\hat{\nu}_0):=\E^{\hat{\pi}}_{\hat{\nu}_0}\left[\sum_{t=0}^{T-1}\beta^t\wh{c}(\mu^{1:M}_t,\theta^{1:M}_t)\right],\quad\hat{\pi}\in\wh{\Pi}^{N,T}\\
  &\qquad\text{and}\quad\wh{J}(\hat{\pi},\hat{\nu}_0):=\E^{\hat{\pi}}_{\hat{\nu}_0}\left[\sum_{t=0}^\infty\beta^t\wh{c}(\mu^{1:M}_t,\theta^{1:M}_t)\right],\quad\hat{\pi}\in\wh{\Pi}^{N}
\end{align*}
respectively, where $\hat{\nu}_0=\nu_0\circ(\mu^{1:M})^{-1}$ is the pushforward
of $\nu_0$ under $\mu^{1:M}:\ess\rightarrow\wh{\mbb{X}}$, and $\wh{c}$
is as described in Lemma \ref{lem1}. Our objective is to obtain a policy
$\hat{\pi}^\ast$ which is optimal in the sense that
\begin{align}
  \wh{J}_T^\ast(\hat{\nu}_0):=\wh{J}_T(\hat{\pi}^\ast,\hat{\nu}_0)=\inf_{\pi\in\wh{\Pi}^{N,T}}\wh{J}_T(\pi,\hat{\nu}_0)\quad\text{and}\quad \wh{J}^\ast(\hat{\nu}_0):=\wh{J}(\hat{\pi}^\ast,\hat{\nu})=\inf_{\pi\in\wh{\Pi}^N}\wh{J}(\pi,\hat{\nu}_0)
\end{align}
with the provisio that $\hat{\pi}^\ast\in\wh{\Pi}^{N,T}$ in the former case and
$\hat{\pi}^\ast\in\wh{\Pi}^N$ in the latter. From the discussion thus far, we
conclude that
\begin{align}
  \left(\wh{\mbb{X}}^N,\wh{\mbb{U}}^N,\big\{\wh{\mbb{U}}^N(\mu^{1:M}):\mu^{1:M}\in\wh{\mbb{X}}\big\},\wh{\mc{T}}^N,\wh{c}, \hat{\nu}_0\right)
\end{align}
constitutes an MDP, which we shall denote with $\wh{\mf{G}}^N_T$ in the
finite horizon case, $\wh{\mf{G}}^N_\infty$ for the infinite horizon, or
$\wh{\mf{G}}^N$ when referring to either interchangeably.

\subsection{Optimality of the cCMF IS, and C-Ex\texorpdfstring{$^N$}{TEXT} Policies for \texorpdfstring{$\mf{G}^N$}{TEXT}}\label{sec6}
According to the above construction, dynamic programming recursions for
$\wh{\mf{G}}^N_T$ are given by
\begin{align}
  &\wh{J}^N_{T-1,T}(\mu^{1:M}):=\inf_{\theta^{1:M}\in\wh{\mbb{U}}(\mu^{1:M})}\left\{\wh{c}(\mu^{1:M},\theta^{1:M})\right\}\label{eq25}\\
  &\wh{J}^N_{t,T}(\mu^{1:M}):=\inf_{\theta^{1:M}\in\wh{\mbb{U}}(\mu^{1:M})}\left\{\wh{c}(\mu^{1:M},\theta^{1:M})+\beta\int_{\wh{\mbb{X}}}\wh{J}^N_{t+1,T}(\mu)\wh{\mc{T}}^N(d\mu|\mu^{1:M},\theta^{1:M})\right\}\label{eq26}
\end{align}
where $0\leq t\leq T-2$ and $\mu^{1:M}\in\wh{\mbb{X}}$. For the well-definedness
of recursions (\ref{eq25}-\ref{eq26}) for all $T\geq 1$ (and thus the existence
of an optimal policy for $\wh{\mf{G}}^N_T$) see appendix A Theorem \ref{thm2}
(which establishes measurable selection conditions via weak continuity of
$\wh{\mc{T}}^N$ under assumption \ref{ass1}). Then, by the monotone convergence
theorem and standard contraction argument for the operator realizing these
recursions, the limit
\begin{align}
  \wh{J}^N_\infty(\mu^{1:M}):=\lim_{T\rightarrow\infty}\wh{J}^N_{t,T}(\mu^{1:M})
\end{align}
exists and satisfies the discounted cost optimality equation 
\begin{align}
  \wh{J}^N_{\infty}(\mu^{1:M})=\inf_{\theta^{1:M}\in\wh{\mbb{U}}(\mu^{1:M})}\left\{\wh{c}(\mu^{1:M},\theta^{1:M})+\beta\int_{\wh{\mbb{X}}}\wh{J}^N_\infty(\mu)\wh{\mc{T}}^N(d\mu|\mu^{1:M},\theta^{1:M})\right\}\label{eq46}
\end{align}
which, again by the satisfaction of measurable selection conditions, ensures the
existence of an optimal stationary policy $\hat{\gamma}\in\wh{\Gamma}^N$ for
$\wh{\mf{G}}^N_\infty$ (for a detailed exposition on these classical arguments,
see \cite[$\S$ 4.2]{Hernandez-Lerma_Lasserre_1996}).\\[5pt]
\indent We are finally ready to establish the first of two main results for this
section: that $\wh{\mf{G}}^N$ is equivalent to $\mf{G}^N$ by viewing the cluster
empirical distributions as a sufficient statistic for the underlying ensemble
activity. The arguments used here follow \cite[theorem 3.3]{Bauerle_2023} and
the supporting discussion therein, but for clarity make more explicit invocation
of Blackwell's classical theorem on irrelevant information for MDPs
\cite{Blackwell_1964}.
\begin{theorem}[Equivalence of $\wh{\mf{G}}^N$ and $\mf{G}^N$]\label{thm5}
  Suppose assumption \ref{ass1} holds. Then without loss of optimality, policies
  for $\mf{G}^N$ may be restricted to those which are Markov and depend on the
  cCMF IS given by
  \begin{align*}
    \ccmfis=\{\mu^{1:M}[x^{1:N}_t]\},\quad t\geq 0.
  \end{align*} 
  Moreover, both $\wh{J}^N_{t,T}(\mu^{1:M}[x^{1:N}])=J^N_{t,T}(x^{1:N})$ and
  $\wh{J}^N_\infty(\mu^{1:M}[x^{1:N}])=J^N_\infty(x^{1:N})$ for all
  $x^{1:N}\in\ess$.
\end{theorem}
\noindent{\bf Proof}. Fixing a deterministic Markovian policy
$\gamma=\gamma_{0:T-1}\in\bar{\Gamma}^{N,T}$ (which is itself without loss
thanks to \cite{Blackwell_1964}), lemma \ref{lem1} certifies that
$\mu^{1:M}[(x^{1:N}_t,u^{1:N}_t)]\in\wh{\mbb{U}}(\mu^{1:M}[x^{1:N}_t])$ (setting
$u^{1:N}_t=\gamma^{1:N}_t(x^{1:N}_t)$ for $0\leq t\leq T-1$) satisfies
\begin{align}
  c({x^{1:N}_t,u^{1:N}_t})=\wh{c}(\mu^{1:M}[x^{1:N}_t],\mu^{1:M}[(x^{1:N}_t,u^{1:N}_t)]),\quad 0\leq t\leq T-1.\label{eq35}
\end{align}
Any such measurable mapping $(x^{1:N}_t,\mu^{1:M}[x^{1:N}_t])\mapsto
\mu^{1:M}[(x^{1:N}_t,\gamma^{1:N}_t(x^{1:N}_t))]$ can be expressed with a policy
$\hat{\gamma}_t:\bar{\mbb{X}}^N\times\wh{\mbb{X}}^N\rightarrow\wh{\mbb{U}}^N$,
which definitionally satisfies the action constraint of $\wh{\mf{G}}^N_T$.
Then Blackwell's theorem certifies that for any
$P_{T-1}\in\mc{P}(\ess\times\mss)$ with marginal $Q_{T-1}$ on $\mss$, there
exists another policy $\hat{\gamma}^\ast_{T-1}:\mss\rightarrow\mas$ such that
\begin{align}
  \int_{\ess\times\wh{\mbb{X}}}\wh{c}(\mu^{1:M},&\hat{\gamma}_{T-1}(x^{1:N},\mu^{1:M}))P_{T-1}(dx^{1:N}\times d\mu^{1:M})\notag\\
  &\quad\qquad\geq \int_{\wh{\mbb{X}}}\wh{c}(\mu^{1:M},\hat{\gamma}^\ast_{T-1}(\mu^{1:M}))Q_{T-1}(d\mu^{1:M})\label{eq99}
\end{align}
so in fact, we may (without loss) replace the policy $\hat{\gamma}_{T-1}$
realizing the cost in (\ref{eq35}) at $T-1$ with $\hat{\gamma}^\ast_{T-1}$,
which only relies on the ensemble up to the empirical distribution of each
cluster. Moreover,
\begin{align}
  J^N_{T-1,T}(x^{1:N}_{T-1})&=\inf_{u^{1:N}\in\eas}\left\{c^N(x^{1:N}_{T-1},u^{1:N})\right\}\notag\\
  &=\inf_{\theta^{1:M}\in\wh{\mbb{U}}(\mu^{1:M}[x^{1:N}_{T-1}])}\left\{\wh{c}(\mu^{1:M}[x^{1:N}_{T-1}],\theta^{1:M})\right\}=\wh{J}^N_{T-1,T}(\mu^{1:M}[x^{1:N}])\label{eq36}
\end{align}
since (by lemma \ref{lem1}) the infima are taken over the same sets. Proceeding
backwards, (and repeating the arguments used in (\ref{eq35}-\ref{eq36})) observe
that
\begin{align}
  &c^N(x^{1:N}_{T-2},u^{1:N}_{T-2})+\beta\int_{\ess}J^N_{T-1,T}(x)\mc{T}^N(dx|x^{1:N}_{T-2},u^{1:N}_{T-2})\notag\\
  &=\wh{c}(\mu[x^{1:N}_{T-2}],\hat{\gamma}_{T-2}(x^{1:N}_{T-2},\mu^{1:M}[x^{1:N}_{T-2}]))\cdots\notag\\
  &\qquad\qquad\qquad\qquad\cdots+\beta\int_{\wh{\mbb{X}}}\wh{J}^N_{T-1,T}(\mu)\wh{\mc{T}}^N(d\mu|\mu^{1:M}[x^{1:N}_{T-2}],\hat{\gamma}_{T-2}(x^{1:N}_{T-2},\mu^{1:M}[x^{1:M}_{T-2}])).\label{eq37}
\end{align}
Invoking Blackwell's theorem \cite{Blackwell_1964}, we may once more replace $\hat{\gamma}_{T-2}$ with
a new policy $\hat{\gamma}_{T-2}^\ast:\mss\rightarrow\mas$ depending only on the
ensemble $x^{1:N}_{T-2}$ up to cluster empirical equivalence. Furthermore,
\begin{align*}
  J^N_{T-2,T}(x^{1:N}_{T-2})&=\wh{J}^N_{T-2,T}(\mu^{1:M}[x^{1:N}_{T-2}]).
\end{align*}
Thus, by backward induction on $\{0,1,2,\dots,T-1\}$, we get
$J^N_{t,T}(x^{1:N}_t)=\wh{J}^N_{t,T}(\mu^{1:M}[x^{1:N}_t])$ for $0\leq t\leq
T-1$, where each minimum is achieved by a deterministic Markov policy
$\hat{\gamma}^\ast_t:\mss\rightarrow\mas$ without loss.\\[5pt]
\indent For $\mf{G}^N_\infty$, observe that the above arguments hold for any
$T\geq 1$ so by a fixed point argument with respect to the Bellman operator we
immediately obtain
\begin{align}
  \wh{J}^N_\infty(\mu^{1:M}[x^{1:N}])=\lim_{T\rightarrow\infty}\wh{J}^N_{t,T}(\mu^{1:M}[x^{1:N}])=\lim_{T\rightarrow\infty}J^N_{t,T}(x^{1:N})=J^N_\infty(x^{1:N}),\quad \forall\:x^{1:N}\in\ess
\end{align}
for any $t\geq 0$ fixed. Then, by the same arguments as used in (\ref{eq37}), we
can show that a policy satisfying the discounted cost optimality equation for
$\mf{G}^N_\infty$ depends only on the cluster empirical distributions, and thus
gives rise to an optimal stationary policy utilizing only the cCMF IS.
\hfill{$\qed$}\\[5pt]
\indent Theorem \ref{thm5} shows that we can view $\wh{\mf{G}}^N$ as a
sufficient representation of $\mf{G}^N$, which immediately yields the optimality
of the cCMF IS for $\mf{G}^N$. While interesting in its own right, the
significance of Theorem \ref{thm5} to our infinite population analysis is solely
mediated through the following result, which establishes the optimality of
C-Ex$^N$ policies for $\mf{G}^N$.
\begin{theorem}[Optimality of Cluster-Exchangeable Policies for
$\mf{G}^N$]\label{thm6} Let assumption \ref{ass1} hold. Then for any $N\geq M$
and $\pi=\pi_{0:T-1}\in\bar{\Pi}^{N,T}$ ($\pi\in\bar{\Pi}^{N}$) there exists a
policy $\hat{\pi}=\hat{\pi}_{0:T-1}\in\Pi^{N,T}_{\cex}$ for $\mf{G}^N_T$
($\pi^\ast=\pi^{\ast}_{0:\infty}\in\Pi^{N}_{\cex}$ for $\mf{G}^N_\infty$)
such that $J^N_T(\pi,\nu_0)=J^N_T(\hat{\pi},\nu_0)$
($J^N(\pi,\nu_0)=J^N(\hat{\pi},\nu_0)$).
\end{theorem}
\noindent{\bf Proof}. We will consider $\mf{G}^N_T$ first. Let
$\pi=\pi_{0:T-1}\in\bar{\Pi}^{N,T}$ be a Markovian policy, and for $0\leq t<T$
define $J^{N,\pi}_{t,T}:\ess\rightarrow\mbb{R}_+$ to be the expected cost-to-go under
$\pi$ at timestep $t$. For $(x^{1:N},u^{1:N})\in\ess\times\eas$, define the
uniformized policy as follows:
\begin{align*}
  \hat{\pi}_t:=\frac{1}{|\mc{S}_{\mc{C}^N}|}\sum_{\sigma\in\mc{S}_{\mc{C}^N}}\pi^\sigma_t,\quad\text{where}\quad \pi^{\sigma}_t(du^{1:N}|x^{1:N})=\pi_t(du^{\sigma(1):\sigma(N)}|x^{\sigma(1):\sigma(N)}),\quad\forall\,\sigma\in\mc{S}_{\mc{C}^N}
\end{align*}
such that $\hat{\pi}:=\hat{\pi}_{0:T-1}\in\Pi^{N,T}_{\cex}$, and
$J^{N,\hat{\pi}}_{t,T}:\ess\rightarrow\mbb{R}_+$ are the associated expected cost-to-go
functions. For $0\leq t<T$, let $P^{\hat{\pi}}_{\nu_0,t}\in\mc{P}(\ess)$ be the
marginal law of $x^{1:N}_t$ under $\hat{\pi}$, and further define
$P_t(dx^{1:N}\times
du^{1:N}):=\pi_t(du^{1:N}|x^{1:N})P^{\hat{\pi}}_{\nu_0,t}(dx^{1:N})$, the
marginal law of $(x^{1:N}_t,u^{1:N}_t)$ for $0\leq t<T$ by replacing
$\hat{\pi}_t$ with $\pi_t$. From this, define the uniformized marginals
\begin{align}
  \wh{P}_t:=\frac{1}{|\mc{S}_{\mc{C}^N}|}\sum_{\sigma\in\mc{S}_{\mc{C}^N}}P^\sigma_t,\quad\text{where}\quad P^\sigma_t(dx^{1:N}\times du^{1:N}):=P_t(dx^{\sigma(1):\sigma(N)}\times du^{\sigma(1):\sigma(N)}),\quad\forall\,\sigma\in\mc{S}_{\mc{C}^N}\notag
\end{align} 
for each $0\leq t<T$. Because $P^{\hat{\pi}}_{\nu_0,t}$ is C-Ex$^N$ for $0\leq t<T$ (where the case of $t=0$ follows since $\nu_0$ itself is C-Ex$^N$) we have
\begin{align}
  \wh{P}_t(dx^{1:N}\times du^{1:N})=\bigg(\frac{1}{|\mc{S}_{\mc{C}^N}|}\sum_{\sigma\in\mc{S}_{\mc{C}^N}}\pi_t(du^{\sigma(1):\sigma(N)}|x^{\sigma(1):\sigma(N)})\bigg)P^{\hat{\pi}}_{\nu_0,t}(dx^{1:N}),\quad 0\leq t<T
\end{align}
such that $\wh{P}_t$ and $P_t$ share the same marginal on $\ess$ for all $0\leq t<T$. We proceed
inductively, showing that $\hat{\pi}_t$ can be replaced with $\pi_t$ at each
stage without altering the cost-to-go. Beginning at time $T-1$, let $A\in\mc{B}(\ess)$. Then
\begin{align}
  &\int_{A\times\eas}c^N(x^{1:N}_{T-1},u^{1:N}_{T-1})\wh{P}_{T-1}(dx^{1:N}_{T-1}\times du^{1:N}_{T-1})\notag\\
  &\quad=\frac{1}{|\mc{S}_{\mc{C}^N}|}\sum_{\sigma\in\mc{S}_{\mc{C}^N}}\int_{A\times\eas}\wh{c}(\mu^{1:M}[x^{1:N}_{T-1}],\mu^{1:M}[(x^{1:N}_{T-1},u^{1:N}_{T-1})])P^\sigma_{T-1}(dx^{1:N}_{T-1}\times du^{1:N}_{T-1})\label{eq117}\\
  &\quad=\frac{1}{|\mc{S}_{\mc{C}^N}|}\sum_{\sigma\in\mc{S}_{\mc{C}^N}}\int_{A\times\eas}\wh{c}(\mu^{1:M}[x^{\sigma(1):\sigma(N)}_{T-1}],\mu^{1:M}[(x^{\sigma(1):\sigma(N)}_{T-1},u^{\sigma(1):\sigma(N)}_{T-1})])P^\sigma_{T-1}(dx^{1:N}_{T-1}\times du^{1:N}_{T-1})\notag\\
  &\quad=\frac{1}{|\mc{S}_{\mc{C}^N}|}\sum_{\sigma\in\mc{S}_{\mc{C}^N}}\int_{A\times\eas}c(x^{1:N}_{T-1},u^{1:N}_{T-1})P_{T-1}(dx^{1:N}_{T-1}\times du^{1:N}_{T-1})\notag\\
  &\quad=\int_{A\times\eas}c^N(x^{1:N}_{T-1},u^{1:N}_{T-1})P_{T-1}(dx^{1:N}_{T-1}\times du^{1:N}_{T-1})\notag
\end{align}
where (\ref{eq117}) follows the definition of $\wh{P}_{T-1}$ and lemma
\ref{lem1}. In particular, the above shows both:
\begin{align}
  &\E[J^{N,\pi}_{T-1,T}(x^{1:N}_{T-1})|x^{1:N}_{T-1}]=\E[J^{N,\hat{\pi}}_{T-1,T}(x^{1:N}_{T-1})|x^{1:N}_{T-1}],\quad P^{\hat{\pi}}_{\nu_0,T-1}\text{-}\as\quad\text{and}\label{eq119}\\
  &\E[J^{N,\hat{\pi}}_{T-1,T}(x^{\sigma(1):\sigma(N)}_{T-1})|x^{\sigma(1):\sigma(N)}_{T-1}]=\E[J^{N,\hat{\pi}}_{T-1,T}(x^{1:N}_{T-1})|x^{1:N}_{T-1}],\quad P^{\hat{\pi}}_{\nu_0,T-1}\text{-}\as\quad \forall\:\sigma\in\mc{S}_{\mc{C}^N},\label{eq147}
\end{align}
where (\ref{eq147}) is due to the cluster-exchangeability of the cost and policy $\hat{\pi}$. Hypothesizing that (\ref{eq119}-\ref{eq147}) hold at some timestep $1\leq t+1<T$, we apply a
similar sequence of computations to find
\begin{align}
  &\int_{A\times\eas}\bigg(c(x^{1:N}_t,u^{1:N}_t)+\beta\E[J^{N,\hat{\pi}}_{t+1,T}(x^{1:N}_{t+1})|x^{1:N}_t,u^{1:N}_t]\bigg)\wh{P}_t(dx^{1:N}_t\times du^{1:N}_t)\notag\\
  &\quad=\frac{1}{|\mc{S}_{\mc{C}^N}|}\sum_{\sigma\in\mc{S}_{\mc{C}^N}}\int_{A\times\eas}\bigg(\wh{c}(\mu^{1:M}[x^{1:N}_t],\mu^{1:M}[(x^{1:N}_t,u^{1:N}_t)])\cdots\notag\\
  &\cdots+\beta\int_{\ess}\E[J^{N,\hat{\pi}}_{t+1,T}(x^{1:N})|x^{1:N}]\mc{T}^N(dx^{1:N}|x^{1:N}_t,u^{1:N}_t)\bigg)P^\sigma_t(dx^{1:N}_t\times du^{1:N}_t)\notag\\
  &\quad=\frac{1}{|\mc{S}_{\mc{C}^N}|}\sum_{\sigma\in\mc{S}_{\mc{C}^N}}\int_{A\times\eas}\bigg(\wh{c}(\mu^{1:M}[x^{\sigma(1):\sigma(N)}_t],\mu^{1:M}[(x^{\sigma(1):\sigma(N)}_t,u^{\sigma(1):\sigma(N)}_t)])\cdots\notag\\
  &\cdots+\beta\int_{\ess}\E[J^{N,\hat{\pi}}_{t+1,T}(x^{\sigma(1):\sigma(N)})|x^{\sigma(1):\sigma(N)}]\mc{T}^N(dx^{\sigma(1):\sigma(N)}|x^{\sigma(1):\sigma(N)}_t,u^{\sigma(1):\sigma(N)}_t)\bigg)P^\sigma_t(dx^{1:N}_t\times du^{1:N}_t)\notag\\
  &\quad=\frac{1}{|\mc{S}_{\mc{C}^N}|}\sum_{\sigma\in\mc{S}_{\mc{C}^N}}\int_{A\times\eas}\bigg(c(x^{1:N}_t,u^{1:N}_t)\cdots\notag\\
  &\cdots+\beta\int_{\ess}\E[J^{N,\hat{\pi}}_{t+1,T}(x^{1:N})|x^{1:N}]\mc{T}^N(dx^{1:N}|x^{1:N}_t,u^{1:N}_t)\bigg)P_t(dx^{1:N}_t\times du^{1:N}_t)\notag\\
  &\quad=\int_{A\times\eas}\bigg(c(x^{1:N}_t,u^{1:N}_t)+\beta\E[J^{N,\pi}_{t+1,T}(x^{1:N}_{t+1})|x^{1:N}_t,u^{1:N}_t]\bigg)P_t(dx^{1:N}_t\times du^{1:N}_t).\notag
\end{align}
To elaborate, the first equality follows the definition of $\wh{P}_{t}$, iterated
expectations and lemma \ref{lem1}, the second equality follows the
cluster-permutation invariance of the kernel $\mc{T}^N$, the empirical measures
and the cost-to-go (in the sense of (\ref{eq147})), the third equality
follows a change of variables, and the final equality is due to the inductive
hypothesis. Comparing the first and final expressions yields
\begin{align}
  &\E[J^{N,\pi}_{t,T}(x^{1:N}_{t})|x^{1:N}_{t}]=\E[J^{N,\hat{\pi}}_{t,T}(x^{1:N}_{t})|x^{1:N}_{t}],\quad P^{\hat{\pi}}_{\nu_0,t}\text{-}\as,\quad\text{and}\label{eq121}\\
  &\E[J^{N,\hat{\pi}}_{t,T}(x^{\sigma(1):\sigma(N)}_t)|x^{\sigma(1):\sigma(N)}_t]=\E[J^{N,\hat{\pi}}_{t,T}(x^{1:N}_t)|x^{1:N}_t],\quad P^{\hat{\pi}}_{\nu_0,t}\text{-}\as\quad\forall\:\sigma\in\mc{S}_{\mc{C}^N}\label{eq148}.
\end{align}
By backwards induction on $\{0,1,\dots,T-1\}$, (\ref{eq121}-\ref{eq148}) holds for all
$0\leq t<T$ and, since $\nu_0$ is itself C-Ex$^N$, we obtain
\begin{align}
  J^N_T(\pi,\nu_0)=\E_{\nu_0}[J^{N,\pi}_{0,T}(x^{1:N}_0)]=\E_{\nu_0}[J^{N,\hat{\pi}}_{0,T}(x^{1:N}_0)]=J^N_T(\hat{\pi},\nu_0).\label{eq125}
\end{align}
Hence, restriction to $\Pi^{N,T}_{\cex}$ is without loss for $\mf{G}^N_T$
and the law of $(x^{1:N}_t,u^{1:N}_t)$ is cluster-exchangeable without loss for
$0\leq t<T$. The result for $\mf{G}^N_\infty$ follows similarly: fixing
$\pi\in\bar{\Pi}^N$, (\ref{eq125}) holds for any $T\geq 1$, and thus
$J^N(\pi,\nu_0)=J^N(\hat{\pi},\nu_0)$ by the monotone convergence
theorem.\hfill{$\qed$}

\section{Infinite Population Team and Structure of Optima}\label{sec7}
We have seen that optimal policies for $\mf{G}^N$ when $N<\infty$ can be
randomized, C-Ex$^N$, and depend on the centralized IS only up clusters'
empirical distributions without loss. The main objective for this section is to
extend this optimality to $\mf{G}^\infty$. Consequently, a De Finetti-type
representation theorem will permit a characterization of C-Ex policies as those
which are C-Sym, conditionally independent, and genuinely decentralized under
the dCMF IS. Moreover, we will construct a mean-field MDP admitting a
well-defined dynamic program for realizing these decentralized optimal policies.
This is in contrast to the finite population $\mf{G}^N$, where no such
representation was available for C-Ex$^N$ policies. We proceed by first
reviewing the essential definitions/convergence results from the theory of
exchangeable processes, but recast for the more general setting of
cluster-exchangeability. 

\subsection{Infinitely Cluster-Exchangeable Processes}\label{sec8}
We begin with the relevant facts and definitions when there is just one cluster.
Let $Y=(Y^1,Y^2,\dots)$ be a random, $\mbb{Y}$-valued process and $\lambda$ a
random element of $\mc{P}(\mbb{Y})$ with
$\mc{L}(\lambda)=:\Lambda\in\mc{P}(\mc{P}(\mbb{Y}))$. Then $Y$ is called
conditionally $\iid$ with directing measure $\lambda$ if and only if
\begin{align*}
  P(Y^1\in A^1,Y^2\in A^2,\dots|\lambda)=\prod_{i=1}^\infty\lambda(A^i),\quad\text{or}\quad P(Y^1\in A^1,Y^2\in A^2,\dots)=\int_{\mc{P}(\mbb{X})}\prod_{i=1}^\infty\lambda(A^i)\Lambda(d\lambda)
\end{align*}
where $A^i\in\mc{B}(\mbb{Y})$ $\forall i\in \mbb{N}$. Note that $\lambda$ is
nothing more than the regular conditional distribution of each $Y^i$ given
$\sigma(\lambda)$. The classical result of De Finetti holds that infinite
exchangeability of $Y$ is equivalent to the existence of a directing measure
$\lambda$, unique almost surely, satisfying the above representation
\cite[theorem 1.1]{Kallenberg_2005}. Moreover, an analogous characterization
fails for finite exchangeable collections \cite{Diaconis_Freedman_1980}. The
following provides a cursory link between the finite and infinite collections
whenever $Y$ is known to be infinitely exchangeable.
\begin{lemma}[{\cite[lemma 2.15]{Aldous_Ibragimov_1985}}]\label{lem10} Let
  $Y=(Y^1,Y^2,\dots)$ be conditionally $\iid$ with directing measure $\lambda$.
  Then $\mu[Y^{1:N}]\rightarrow\lambda$ $\as$ as $N\rightarrow\infty$.
\end{lemma}
\noindent\textbf{Proof}. Let
$\Xi=\{\lim_{N\rightarrow\infty}\mu[Y^{1:N}]=\lambda\}$ and note
\begin{align}
  P(\Xi)=\E[\E[\2{\Xi}|\lambda]]=\E\bigg[\int_{\mbb{Y}}\1{\lim_{N\rightarrow\infty}\mu[y^{1:N}]=\lambda}\prod_{i=1}^\infty\lambda(dy^i)\bigg]=1\label{eq131}
\end{align}
where the last equality holds by the pointwise ergodic theorem since $Y$ is
$\iid$ given $\lambda$.\hfill{$\qed$}\\[5pt]
\indent In particular, whenever $Y$ is known to be infinitely exchangeable, the
empirical measures $\mu[Y^{1:N}]$ converge weakly $\as$ to the directing measure
$\lambda$.\\[5pt]
\indent In service of infinite population control, such a representation (when
applied to an exchangeable state-action profile at some fixed time point)
provides an avenue for relaxing the information required at each decision maker
since the directing measure is sufficient to render the activities of agents
symmetric and independent \cite{Sanjari_Saldi_Yüksel_2025}. To facilitate a
similar analysis for our model, we have the following De Finetti-type
decomposition for C-Ex processes.
\begin{proposition}[{\cite[proposition 3.8, corollary
  3.9]{Aldous_Ibragimov_1985}}]\label{prop4} Let
  $Y=(Y^{\mc{C}^\infty_1},Y^{\mc{C}^\infty_2},\dots,Y^{\mc{C}^\infty_M})$ be
  C-Ex, $\mbb{Y}^j$-valued over cluster $j\in V$. Denote with $\lambda^j$ the
  directing measure of $Y^{\mc{C}^\infty_j}$, and set
  $\mc{F}:=\sigma(\lambda^j:j\in V)$. Then $Y$ is conditionally independent
  given $\mc{F}$, and for all $i\in\mbb{N}$, $\lambda^{C(i)}$ is a regular
  conditional distribution for $Y^{i}$ given $\sigma(\lambda^{C(i)})$.
  Equivalently,
  \begin{align*}
    &P(Y^i\in A^i,\;\forall i\in\mbb{N}|\mc{F})=\prod_{j\in V}\prod_{i\in\mc{C}^\infty_j}\lambda^j(A^i),
    \\&\text{or}\quad P(Y^i\in A^i,\;\forall i\in\mbb{N})=\int_{\mc{P}(\mbb{Y}^1)\times\mc{P}(\mbb{Y}^2)\times\cdots\times\mc{P}(\mbb{Y}^M)}\prod_{j\in V}\prod_{i\in\mc{C}^\infty_j}\lambda^j(A^i)\Lambda(d\lambda^1\times d\lambda^2,\times\cdots\times d\lambda^M)
  \end{align*}
  where $A^i\in\mc{B}(\mbb{Y}^{C(i)})$ for all $i\in\mbb{N}$, and
  $\mc{L}(\lambda^1,\lambda^2,\dots,\lambda^M)=\Lambda$.
\end{proposition}
\noindent{\bf Proof}. For $j\in V$, define
$Y^{-\mc{C}_j^\infty}:=(Y^{\mc{C}^\infty_1},Y^{\mc{C}^\infty_2},\dots,Y^{\mc{C}^\infty_{j-1}},Y^{\mc{C}^\infty_{j+1}},\dots,Y^{\mc{C}^\infty_M})$
and write $Y^{\mc{C}^\infty_j}=(Y^1_j,Y^2_j,\dots)$. Introduce the new process
$Z_j=(Z^1_j,Z^2_j,\dots)$, with $Z^i_j:=(Y^{-\mc{C}^\infty_j},Y^i_j)$ for all
$i\in\mbb{N}$. Then $Z_j$ is infinitely exchangeable, so by De Finetti's theorem
it is conditionally $\iid$ with some directing measure $\eta^j$. In particular,
$Y^{\mc{C}^\infty_j}$ is itself conditionally independent given
$\sigma(\eta^j)$, but it is also C-Ex in its own right, and hence admits a
directing measure $\lambda^j$. By Lemma \ref{lem10} we then find
\begin{align*}
  \mu[Z_j^{1:N}]=\frac{1}{N}\sum_{i=1}^N\delta_{(Y^{-\mc{C}^\infty_j}, Y^i_j)}&=\delta_{Y^{-\mc{C}^\infty_j}}\otimes\bigg(\frac{1}{N}\sum_{i=1}^N\delta_{Y^i_j}\bigg)
  \overunderset{\as}{N\rightarrow\infty}{\longrightarrow}\delta_{Y^{-\mc{C}^\infty_j}}\otimes\lambda^j.
\end{align*}
Hence, $\eta^j=\delta_{Y^{-\mc{C}^\infty_j}}\otimes\lambda^j$ $\as$, and hence
$\sigma(\eta^j)=\sigma(Y^{-\mc{C}^\infty_j},\lambda^j)$. This implies $\eta^j$
and $Y^{\mc{C}^\infty_j}$ are independent given $\lambda^j$ (see \cite[lemma
2.12]{Aldous_Ibragimov_1985}). That is,
\begin{align*}
  P(Y^i_j\in A^i,\;\forall i\in\mbb{N}\,|\,Y^{-\mc{C}^\infty_j},\lambda^j)=P(Y^i_j\in A^i,\;\forall i\in\mbb{N}\,|\,\lambda^j)
\end{align*}
so $Y^{\mc{C}^\infty_j}$ is conditionally independent of $Y^{-\mc{C}^\infty_j}$
given $\lambda^j$. Since this holds for all $j$, and
$\sigma(\lambda^j)\subseteq\mc{F}\subseteq \sigma(\eta^j:j\in V)$ for $j\in V$,
take $A^i\in\mc{B}(\mbb{Y}^{C(i)})$ for all $i\in\mbb{N}$ and compute
\begin{align*}
  P(Y^i\in A^i,\;\forall i\in\mbb{N}\,|\,\mc{F})&=\prod_{j\in V}P(Y^i\in A^i,\;\forall i\in\mc{C}^\infty_j\,|\,\mc{F})\\&=\prod_{j\in V}P(Y^i\in A^i,\;\forall i\in\mc{C}^\infty_j\,|\,\lambda^j)=\prod_{j\in V}\prod_{i\in\mc{C}^\infty_j}\lambda^j(A^i).\tag*{$\qed$} 
\end{align*}
\indent Some discussion is warranted to highlight the utility of proposition
\ref{prop4} in the context of $\mf{G}^N$, $\mf{G}^\infty$. Foremost, policies
corresponding to C-Ex state-action profiles are symmetric and independent within
each subpopulation, and so both $\bar{\Pi}^{T}_{\cex}=\Pi^{T}_{\csym}$ and
$\bar{\Pi}_{\cex}=\Pi_{\csym}$. To sketch the forthcoming arguments, consider
$\mf{G}^\infty_T$. Fix $1\leq t< T$ and suppose
$(x^{1:\infty}_t,u^{1:\infty}_t)$ is C-Ex with
directing measure $\tilde{\theta}^j_t$ on cluster $\mc{C}^\infty_j$ such that
$\tilde{\theta}^j_t(dx\times\mbb{U}^j)=:\tilde{\mu}^j_t(dx)$ $\forall\:j\in V$.
Proposition \ref{prop4} then implies
\begin{align}
  P((x^{1:\infty}_t,u^{1:\infty}_t)\in A|\tilde{\theta}^{1:M}_t)&=\int_A\prod_{j\in V}\prod_{i\in\mc{C}^\infty_j}\tilde{\theta}^{j}_t(dx^i\times du^i)\notag\\
  &=\int_A\prod_{j\in V}\prod_{i\in\mc{C}^\infty_j}\tilde{\pi}^{j}_t(du^i|x^i)\tilde{\mu}^{j}_t(dx^i)\label{eq52}
\end{align}
where $A\in\mc{B}(\bar{\mbb{X}}\times\bar{\mbb{U}})$, and we identify the kernel
$\tilde{\pi}^j_t$ as the policy of each agent $i\in\mc{C}^\infty_j$. Notably,
the $\tilde{\pi}^j_t$ are determined by $\tilde{\mu}^j_t$ and
$\tilde{\theta}^j_t$. Thus, conditioning on
$\{x^{i}_t,\,\tilde{\mu}^{j}_t,\tilde{\gamma}^j_t,\widetilde{\mc{I}}^i_t\}$,
where $\tilde{\gamma}^j_t$ is the mechanism generating $\tilde{\theta}^j_t$ from
some information process $\widetilde{\mc{I}}^i_t$, agents in $\mc{C}^\infty_j$
independently select actions using an identical policy.\\[5pt]
\indent Restricting to the class of C-Ex strategies will be shown later to yield
a mean-field MDP as a controlled flow of state directing measures. However, in
order to establish that this restriction may be taken without loss, we require
some results on the weak convergence of C-Ex$^N$ processes. The supporting
arguments are essentially unchanged from the fully exchangeable case by first
adopting a compatible definition for the directing measure of a C-Ex process. 
\begin{definition}[Directing Measure of Cluster-Exchangeable
  Processes]\label{def4} Let $Y$ be a cluster-exchangeable process, as above.
  Then its directing measure is the array $\lambda^{1:M}$. This is viewed as a
  $\mc{P}(\prod_{j\in V}\mbb{Y}^j)$-valued random variable by setting
  \begin{align*}
    \lambda^{1:M}(A)=\int_A\prod_{j\in V}\lambda^j(dy^j),\quad \forall\:A\in\mc{B}(\prod_{j\in V}\mbb{Y}^j).
  \end{align*}
\end{definition}
To justify this Definition \ref{def4}, consider
$Y=(Y^{\mc{C}^\infty_1},Y^{\mc{C}^\infty_2},\dots,Y^{\mc{C}^\infty_M})$ as in
proposition \ref{prop4}, and write $\mc{C}^\infty_j=\{k^j_m\}_{m\geq 1}$ so that
$k^j_m<k^j_{m+1}$ for all $m\geq 1$, $j\in V$. Then rewriting the representation
in proposition \ref{prop4},
\begin{align}
  P(Y^i\in A^i,\;\forall i\in\mbb{N}|\lambda^{1:M})&=\prod_{j\in V}\prod_{m\geq 1}\lambda^j(A^{k^j_m})=\prod_{m\geq 1}\prod_{j\in V}\lambda^j(A^{k^j_m})\notag\\
  &=P((Y^{k^1_m},Y^{k^2_m},\dots,Y^{k^M_m})\in(A^{k^1_m},A^{k^2_m},\dots,A^{k^M_m}),\;\forall m\in\mbb{N}|\lambda^{1:M}).\label{eq87}
\end{align}
Hence, the process $\breve{Y}:=(\breve{Y}^1,\breve{Y}^2,\dots)$ with
$\breve{Y}^m:=(Y^{k^1_m},Y^{k^2_m},\dots,Y^{k^M_m})$ for $m\geq 1$ is infinitely
exchangeable (in the sense of the original De Finetti theorem) with directing
measure is $\prod_{j\in V}\lambda^j$. Definition \ref{def4} does not imply
clusters are independent; only that their dependence is mediated through the
joint law $\mc{L}(\lambda^{1:M})$, which assigns full measure on the set of
factorizable measures in $\mc{P}(\prod_{j\in V}\mbb{Y}^j)$.\\[5pt]
\indent In case each $\mbb{Y}^j$ is compact, the subset of continuous, bounded functions
\begin{align}
  \big\{g\in C_b\big(\prod_{j\in V}\mbb{Y}^j\big)\setsep g=\prod_{j\in V}g^j,\;g^j\in C_b(\mbb{Y}^j)\;\forall\:j\in V\big\}
\end{align}
is dense by the Stone-Weierstrass theorem. Hence, weak convergence of directing measures is
characterized by weak convergence in each coordinate, and any such limit is
itself a directing measure \cite[lemma 1.1]{Parthasarathy_1967}. From this we
immediately obtain a generalization of Lemma \ref{lem10} to C-Ex processes.
\begin{lemma}\label{lem8} Let
  $Y=(Y^{\mc{C}^\infty_1},Y^{\mc{C}^\infty_2},\dots,Y^{\mc{C}^\infty_M})$ be
  infinitely cluster-exchangeable, $\mbb{Y}^j$-valued over cluster $j\in V$, and
  suppose these spaces are compact. If $Y$ is directed by $\lambda^{1:M}$, and
  $\lambda^j_N:=\mu^j[Y^{1:N}]$ for $j\in V$, $N\geq M$, then
  $\lambda^{1:M}_N\rightarrow\lambda^{1:M}$ $\as$ as $N\rightarrow\infty$.
\end{lemma}
\noindent{\bf Proof}. Invoke the pointwise ergodic theorem as in Lemma
\ref{lem10} to obtain $\lambda^j_N\rightarrow\lambda^j$ weakly $\as$
$\forall\,j\in V$. The joint weak convergence of
$\lambda^{1:M}_N\rightarrow\lambda^{1:M}$ $\as$ follows since each $\mbb{Y}^j$
is compact.\hfill{$\qed$}\\[5pt]
\indent Lemma \ref{lem8} is applicable whenever we are operating in a `top-down'
regime -- that is, when the infinite ensemble is known to be C-Ex \textit{a
priori}. Conversely, in the `bottom-up' regime where we only know each finite
ensemble is C-Ex$^N$, almost sure convergence of the finite ensemble empirical
measures is lost. Instead, Theorems \ref{thm8} and \ref{thm12} below establish
their convergence in law along a subsequence. The first parallels \cite[Theorem
13]{Diaconis_Freedman_1980}, which associates C-Ex$^N$ collections with C-Ex
ones by comparing sampling with and without replacement.
\begin{theorem}\label{thm8} Fix $N\geq M$, and let
  $\bar{Y}=(\bar{Y}^{\mc{C}^N_1},\bar{Y}^{\mc{C}^N_2},\dots,\bar{Y}^{\mc{C}^N_M})$
  be C-Ex$^N$. Then $\exists\:Y$ C-Ex such that for any positive integers
  $k_j\leq N_j$ for $j\in V$ we have
  \begin{align}
    \|\mc{L}(\bar{Y}^{\mc{C}^\prime_1},\bar{Y}^{\mc{C}^\prime_2},\dots,\bar{Y}^{\mc{C}^\prime_M})-\mc{L}(Y^{\mc{C}^\prime_1},Y^{\mc{C}^\prime_2},\dots,Y^{\mc{C}^\prime_M})\|_{\text{TV}}\leq1-\prod_{j\in V}\bigg(1-\frac{k_j(k_j-1)}{2N_j}\bigg) \notag
  \end{align}
  where $\mc{C}^\prime_j$ consists of the first $k_j$ elements of
  $\mc{C}^\infty_j$ $\forall\:j\in V$.
\end{theorem}
\noindent\textbf{Proof}. The proof follows precisely the same construction as in
\cite[Theorem 13]{Diaconis_Freedman_1980} in light of Definition \ref{def5}. See
Appendix \ref{app2}.
\begin{corollary}\label{cor4}
  Suppose $\mbb{Y}^j$ is compact $\forall\:j\in V$. If $Q_N\in\mc{P}(\prod_{i\in\mc{N}}\mbb{Y}^{C(i)})$ is C-Ex$^N$ $\forall\:N\geq M$, 
  then $\exists\:Q\in\mc{P}(\prod_{i\in\mbb{N}}\mbb{Y}^{C(i)})$ C-Ex and a subsequence $\{N_\ell\}_{\ell\geq 1}$ such that
  \begin{align*}
    Q_{N_\ell}\big|_k\overset{\ell\rightarrow\infty}{\longrightarrow}Q\big|_k,\quad\forall\:k\geq M
  \end{align*}
  where $Q_{N_\ell}\big|_k$, $Q\big|_k$ denote the respective marginals over the $k$-member ensemble.
\end{corollary}
\noindent\textbf{Proof}. Theorem \ref{thm8} grants a sequence of C-Ex extensions
$\{\bar{Q}^N\}_{N\geq M}\subset\mc{P}(\prod_{i\in \mbb{N}}\mbb{Y}^{C(i)})$ so
that
\begin{align*}
  \|Q_N\big|_k-\bar{Q}^N\big|_k\|_{\text{TV}}\overset{N\rightarrow\infty}{\longrightarrow} 0,\quad\forall\:k\geq M.
\end{align*}
Hence, if $\{\bar{Q}^N\}_{N\geq M}$ converges weakly to some
$Q\in\mc{P}(\prod_{i\in\mbb{N}}\mbb{Y}^{C(i)})$, then $\{Q_N\}_{N\geq M}$
converges weakly (along any finite marginal) to $Q$ as well. Since
$\prod_{i\in\mbb{N}}\mbb{Y}^{C(i)}$ is compact, so is
$\mc{P}(\prod_{i\in\mbb{N}}\mbb{Y}^{C(i)})$, and we may extract a subsequence
$\{N_\ell\}_{\ell\geq 1}$ with $\bar{Q}^N\rightarrow Q$ for some
$Q\in\mc{P}(\prod_{i\in\mbb{N}}\mbb{Y}^{C(i)})$ as
$\ell\rightarrow\infty$.\hfill{$\qed$}\\[5pt]
\indent For the second part of our discussion on convergence of empirical
measures, we have the following technical lemma, which generalizes \cite[Lemma
7.14]{Aldous_Ibragimov_1985} to the case of C-Ex collections. Subsequently,
Theorem \ref{thm12} generalizes \cite[Proposition 7.20]{Aldous_Ibragimov_1985}
to characterize weak convergence (along finite marginals) of a C-Ex$^N$ ensemble
in terms of the convergence in law of the corresponding empirical measures. For
Lemma \ref{lem12}, recall a sequence of measures $\{P_n\}_{n\geq
1}\subset\mc{P}(\mbb{Y})$ is said to be tight on $\mbb{Y}$ if and only if
$\forall\:\varepsilon>0$, $\exists\:K_\varepsilon\subset\mbb{Y}$ compact such
that $P_n(K_{\varepsilon})>1-\varepsilon$ $\forall\:n\geq 1$.
\begin{lemma}\label{lem12} Let $\{\lambda^{1:M}_n\}_{n\geq 1}$ be a collection
  of $\mc{P}(\prod_{j\in V}\mbb{Y}^j)$-valued random variables. If
  $\{\E[\lambda^{1:M}_n]\}_{n\geq 1}$ is tight on $\prod_{j\in V}\mbb{Y}^j$,
  then $\{\mc{L}(\lambda^{1:M}_n)\}_{n\geq 1}$ is tight on $\mc{P}(\prod_{j\in
  V}\mbb{Y}^j)$.
\end{lemma}
\noindent\textbf{Proof}. Fix $\varepsilon>0$, and let $K_\ell\subset\prod_{j\in
V}\mbb{Y}^{j}$ be compact such that $\E[\lambda^{1:M}_n(K_\ell)]\geq
1-\varepsilon2^{-2\ell}$ for $\ell\geq 1$. Then, by Markov's inequality:
\begin{align*}
  P(\lambda^{1:M}_n(K^c_\ell)>2^{-\ell})\leq 2^{\ell}\E[\lambda^{1:M}_n(K^c_\ell)]\leq\varepsilon2^{-\ell},\quad\forall\ell\geq 1.
\end{align*}
Now, define $\Theta:=\{\theta\in\mc{P}(\prod_{j\in
V}\mbb{Y}^j):\theta(K^c_\ell)\leq 2^{-\ell}\,\forall\,\ell\geq 1\}$, and suppose
$\{\theta_n\}_{n\geq 1}\subseteq\Theta$ is such that
$\theta_n\rightarrow\theta\in\mc{P}(\prod_{j\in V}\mbb{Y}^j)$ weakly as
$n\rightarrow\infty$. Then, by the portmanteau lemma:
\begin{align*}
  2^{-\ell}\geq \liminf_{n\rightarrow\infty}\theta_n(K^c_\ell)\geq\theta(K^c_{\ell}),\quad\forall\,\ell\geq 1
\end{align*}
so $\theta\in\Theta$, and hence $\Theta$ is closed. It is also weakly precompact
since, for any $\delta>0$, we may take $\ell\geq 1$ sufficiently large such that
$2^{-\ell}<\delta$. In this case, $\theta(K_\ell)\geq 1-2^{-\ell}\geq 1-\delta$
for all $\theta\in\Theta$, so $\Theta$ is compact. Then, $\forall\:n\geq 1$, the
Markov estimate above yields:
\begin{align*}
  P(\lambda^{1:M}_n\in\Theta)=1-P(\exists\:\ell\geq 1\;\text{s.t.}\;\lambda^{1:M}_n(K^c_\ell)>2^{-\ell})\geq 1-\sum_{\ell\geq 1}P(\lambda^{1:M}_n(K^c_\ell)>2^{-\ell})\geq 1-\varepsilon
\end{align*}
such that the laws $\{\mc{L}(\lambda^{1:M}_n)\}_{n\geq 1}$ are tight on
$\mc{P}(\prod_{j\in V}\mbb{Y}^j)$.\hfill{$\qed$}

\begin{theorem}\label{thm12} Let
  $Y=(Y^{\mc{C}^\infty_1},Y^{\mc{C}^\infty_2},\dots,Y^{\mc{C}^\infty_M})$ be
  C-Ex, where $Y^i$ is $\mbb{Y}^{C(i)}$-valued for all $i\in\mbb{N}$, directed
  by $\lambda^{1:M}$. Suppose either:
  \begin{enumerate}[(i)]
    \item
    $Y_n=(Y^{\mc{C}^\infty_1}_n,Y^{\mc{C}^\infty_2}_n,\dots,Y^{\mc{C}^\infty_M}_n)$
    is C-Ex and directed by $\lambda^{1:M}_n$, or
    \item
    $Y_n=(Y^{\mc{C}^{N_n}_1}_n,Y^{\mc{C}^{N_n}_2}_{n},\dots,Y^{\mc{C}^{N_n}_M}_n)$
    is C-Ex$^N$ with empirical measures
    $\lambda^{1:M}_n:=\lambda^{1:M}[Y^{1:N_n}_n]$ and where
    $N_n\rightarrow\infty$ as $n\rightarrow\infty$.
  \end{enumerate}
  Then $\mc{L}(Y_n)\rightarrow \mc{L}(Y)$ along any finite marginal as
  $n\rightarrow\infty$ if and only if
  $\mc{L}(\lambda^{1:M}_n)\rightarrow\mc{L}(\lambda^{1:M})$ as
  $n\rightarrow\infty$.
\end{theorem}
\noindent\textbf{Proof}. Beginning with case (i), suppose
$\mc{L}(\lambda^{1:M}_n)\rightarrow\mc{L}(\lambda^{1:M})$ weakly in
$\mc{P}(\mc{P}(\prod_{j\in V}\mbb{Y}^j))$ as $n\rightarrow\infty$. Let
$\bar{\lambda}^{1:M}_n$ $\bar{\lambda}^{1:M}$ be Skorohod representations of
$\mc{L}(\lambda^{1:M}_n)$, $\mc{L}(\lambda^{1:M})$ $\forall\:n\geq 1$ such that
$\bar{\lambda}^{1:M}_n\rightarrow\bar{\lambda}^{1:M}$ $\as$ as
$n\rightarrow\infty$. Then the C-Ex processees $\bar{Y}_n$ and $\bar{Y}$
directed by $\bar{\lambda}^{1:M}_n$, $\bar{\lambda}^{1:M}$ respectively are such
that $\mc{L}(Y_n)=\mc{L}(\bar{Y}_n)$ and $\mc{L}(Y)=\mc{L}(\bar{Y})$
$\forall\:n\geq 1$. Define the sub-$\sigma$-algebras
$\mc{F}_n:=\sigma(\lambda^{1:M}_n)$ and $\mc{F}:=\sigma(\mu^j:j\in V)$
$\forall\:n\geq 1$. Then, fixing $r\geq 1$ and $f_k\in C_b(\mbb{Y}^{C(k)})$ for
$k\leq r$, we have
\begin{align}
    \E[\prod_{k\leq r}f_k(\bar{Y}^k_n)]&=\E[\prod_{k\leq r}\E[f_k(\bar{Y}^k_n)|\mc{F}_n]]\notag\\
    &=\E[\prod_{k\leq r}\int_{\mbb{Y}^{C(k)}}f_k(y)\lambda^{C(k)}_n(dy)]\overset{n\rightarrow\infty}{\longrightarrow}\E[\prod_{k\leq r}\int_{\mbb{Y}^{C(k)}}f_k(y)\lambda^{C(k)}(dy)]
    =\E[\prod_{k\leq r}f_k(\bar{Y}^k)].\notag
\end{align}
To be precise, the first equality follows iterated expectations, proposition
\ref{prop4} (namely, that $\{\bar{Y}^k_n\}_{k\geq 1}$ are independent given
$\mc{F}_n$) and Fubini-Tonelli. The limit follows dominated convergence in
conjunction with the $\as$ convergence of $\lambda^{1:M}_n$. This class of
functionals characterizes weak convergence of the finite marginals
$\mc{L}(\bar{Y}^1_n,\bar{Y}^2_n,\dots,\bar{Y}^r_n)\rightarrow\mc{L}(\bar{Y}^1,\bar{Y}^2,\dots,\bar{Y}^r)$
(e.g., \cite[theorem 4.29]{Kallenberg_2002}) for any $r\geq 1$, so
$\mc{L}(Y_n)=\mc{L}(\bar{Y}_n)\rightarrow\mc{L}(\bar{Y})=\mc{L}(Y)$ as
$n\rightarrow\infty$.\\[5pt]
\indent For the converse, suppose that $\mc{L}(Y_n){\rightarrow}\mc{L}(Y)$ as
$n\rightarrow\infty$. Then $\{\E[\lambda^{1:M}_n]\}_{n\geq
1}=\{\mc{L}(Y^1_n)\}_{n\geq 1}$ is tight, which, by lemma \ref{lem12}, implies
tightness of $\{\mc{L}(\lambda^{1:M}_n)\}_{n\geq 1}$ on $\mc{P}(\prod_{j\in
V}\mbb{Y}^j)$. Hence, for any subsequence, one can extract sub-subsequence
$\{\lambda^{1:M}_{n_k}\}_{k\geq 1}$ converging in distribution to some
$\lambda^{1:M,\ast}$, whereby the previous part we conclude
$\lambda^{1:M,\ast}=\lambda^{1:M}$. Thus,
$\mc{L}(\lambda^{1:M}_n)\rightarrow\mc{L}(\lambda^{1:M})$ as
$n\rightarrow\infty$.\\[5pt]
\indent For case (ii), construct a C-Ex extension $\tilde{Y}_n$ directed by
empirical measure $\lambda_n^{1:M}$ for $n\geq 1$ as per Theorem \ref{thm8}. By
the same result, it follows that $\mc{L}(Y_n)\rightarrow\mc{L}(Y)$ as
$n\rightarrow\infty$ if and only if $\mc{L}(\tilde{Y}_n)\rightarrow\mc{L}(Y)$ as
$n\rightarrow\infty$, which in turn occurs if and only if
$\mc{L}(\lambda^{1:M}_n)\rightarrow\mc{L}(\lambda^{1:M})$ by case
(i).\hfill{$\qed$}\\[5pt]
\indent The so-called bottom-up convergence described by Theorems \ref{thm8} and \ref{thm12} is key to establishing our main resuts on the
structure/informational demands of optima in the following sections (see Theorems \ref{thm7} and
\ref{thm10} below).
\subsection{Structure of Optima for \texorpdfstring{$\mf{G}^\infty$}{TEXT}: Part I}\label{sec9}
We can now present our first main result on the structure of optima for
$\mf{G}^\infty$. As a preliminary result, the following lemma restates Corollary
\ref{cor4} and highlights the significance of Theorem \ref{thm6} by establishing
that a sequence of optimal, C-Ex$^N$ strategic measures converges
subsequentially to a C-Ex joint measure for the infinite ensemble.
\begin{lemma}\label{lem5} Let assumption \ref{ass1} hold, and take
  $\pi^N\in\Pi^{N,T}_{\cex}$ such that
  $Q_N:=P^{\pi^N}_{\nu_0}\in\mc{P}(\prod_{t=0}^{T-1}\ess\times\eas)$ is C-Ex$^N$
  $\forall\:N\geq M$. Then
  $\exists\:Q\in\mc{P}(\prod_{t=0}^{T-1}\bar{\mbb{X}}\times\bar{\mbb{U}})$ C-Ex
  such that
  \begin{align*}
    Q_{N_\ell}\big|_{k}\overset{\ell\rightarrow\infty}{\longrightarrow}Q\big|_{k}\quad\forall\:k\geq M
  \end{align*}
  for a subsequence $\{N_\ell\}_{\ell\geq 1}$, where $Q_{N_\ell}\big|_k$,
  $Q\big|_k$ denote the respective marginals over the $k$-member ensemble.
\end{lemma}
\noindent\textbf{Proof}. The proof is precisely that of Corollary \ref{cor4}.\hfill{$\qed$}\\[5pt]
\indent An analogous result holds for the infinite horizon case. To see this,
note that for $N\geq M$ fixed, the C-Ex$^N$ random variables in question are
$\prod_{t=0}^{T-1}\ess\times\eas$-valued, but the case of $\prod_{t\geq
0}\ess\times\eas$-valued random variables has no bearing on the application of
Theorem \ref{thm8}.\\[5pt]
\indent Lemma \ref{lem5} along with our preceding discussion is sufficient to
establish the optimality of C-Ex strategic measures for $\mf{G}^\infty$.
\begin{theorem}\label{thm7} Let assumption \ref{ass1} hold. Then
  \begin{align}
    \inf_{\pi\in\bar{\Pi}^{T}}J_T(\pi,\nu_0)=\inf_{\pi\in\Pi^{T}_{\cex}}J_T(\pi,\nu_0),\quad\text{and}\quad \inf_{\pi\in\bar{\Pi}}J(\pi,\nu_0)=\inf_{\pi\in\Pi_{\cex}}J(\pi,\nu_0).
  \end{align}
\end{theorem}
\noindent{\bf Proof}. Consider the finite horizon problem $\mf{G}^\infty_T$
first. By interchanging the limit supremum and infimum we obtain the lower bound
\begin{align}
  J^\ast_T(\nu_0)=\inf_{\pi\in\bar{\Pi}^T}\limsup_{N\rightarrow\infty}J^N_T(\pi\big|_N,\nu_0)&\geq\limsup_{N\rightarrow\infty}\inf_{\pi\in\bar{\Pi}^T}J^N_T(\pi\big|_N,\nu_0)\notag\\
  &=\limsup_{N\rightarrow\infty}J^{N,\ast}_T(\nu_0)\notag\\
  &\geq\liminf_{N\rightarrow\infty}J^{N,\ast}_T(\nu_0)\label{eq135}
\end{align}
whence we may extract a subsequence $\{N_\ell\}_{\ell\geq 1}$ converging to the
limit infimum in (\ref{eq135}). That is,
\[\lim_{\ell\rightarrow\infty}J^{N_\ell,\ast}_T(\nu_0)=\liminf_{N\rightarrow\infty}J^{N,\ast}_T(\nu_0).\]
Now, by Theorem \ref{thm6}, each value function in the lower bound (\ref{eq135})
is realized by a policy $\pi^{N,\ast}\in\Pi^{N,T}_{\cex}$ rendering the
resulting strategic measure
$P^{\pi^{N,\ast}}_{\nu_0}\in\mc{P}(\prod_{t=0}^{T-1}\ess\times\eas)$ C-Ex$^N$.
Thus, we may invoke Lemma \ref{lem5} to extract a further subsequence
(which we will still call $\{N_\ell\}_{\ell\geq 1}$) such that strategic measures
$\{P^{\pi^{N_\ell,\ast}}_{\nu_0}\}_{\ell\geq 1}$ converge weakly to a C-Ex
measure $\wt{P}\in\mc{P}(\prod_{t=0}^{T-1}\bar{\mbb{X}}\times\bar{\mbb{U}})$
along any finite marginal. Subsequently, Theorem \ref{thm12} case (ii) implies
the convergence in law of the empirical measures to those of a collection of
directing measures governing the infinite ensemble. To illustrate, we continue
from (\ref{eq135})
\begin{align}
  J^\ast_T(\nu_0)&\geq\liminf_{\ell\rightarrow\infty}J^{N_\ell}_T(\pi^{N_\ell,\ast},\nu_0)\label{eq152}\\
  &=\liminf_{\ell\rightarrow\infty}\int_{\prod_{t=0}^{T-1}\bar{\mbb{X}}^{N_\ell}\times\bar{\mbb{U}}^{N_\ell}}\sum_{t=0}^{T-1}\beta^tc^{N_\ell}(x^{1:N_\ell}_t,u^{1:N_\ell}_t)P^{\pi^{N_\ell,\ast}}_{\nu_0}(dx^{1:N_\ell}_{0:T-1}\times du^{1:N_\ell}_{0:T-1})\notag\\   
  &=\liminf_{\ell\rightarrow\infty}\int_{\prod_{t=0}^{T-1}\bar{\mbb{X}}^{N_\ell}\times\bar{\mbb{U}}^{N_\ell}}\sum_{t=0}^{T-1}\beta^t\wh{c}(\mu^{1:M}[x^{1:N_\ell}_t],\mu^{1:M}[(x^{1:N_\ell}_t,u^{1:N_\ell}_t)])P^{\pi^{N_\ell,\ast}}_{\nu_0}(dx^{1:N_\ell}_{0:T-1}\times du^{1:N_\ell}_{0:T-1})\notag\\
  &=\liminf_{\ell\rightarrow\infty}\int_{\prod_{t=0}^{T-1}\wh{\mbb{X}}^{N_\ell}\times\wh{\mbb{U}}^{N_\ell}}\sum_{t=0}^{T-1}\beta^t\wh{c}(\mu_t,\theta_t)Q^{\pi^{N_\ell,\ast}}_{\nu_0}(d\mu_{0:T-1}\times d\theta_{0:T-1})\label{eq149}
\end{align}
where $Q^{\pi^{N_\ell,\ast}}_{\nu_0}$ in (\ref{eq149}) is the strategic measure
induced on the empirical measure profile by $\pi^{N_\ell,\ast}$. Further, define
$Q^{N_\ell}_t$ to be the marginal distribution of
$(\mu^{1:M}[x^{1:N_\ell}_t],\mu^{1:M}[(x^{1:N_\ell}_t,u^{1:N_\ell}_t)])$ induced
by $\pi^{N_\ell,\ast}$ $\forall\:\ell\geq 1$ and $0\leq t<T$. Then by Theorem
\ref{thm12} we have $Q^{N_\ell}_t\rightarrow \wt{Q}_t$ as
$\ell\rightarrow\infty$, where $\wt{Q}_t$ is the law of the directing measure of
the subsequential limit ensemble at time $0\leq t<T$. Moreover,
\begin{align}
  J^\ast_T(\nu_0)&\geq\liminf_{\ell\rightarrow\infty}\sum_{t=0}^{T-1}\beta^t\int_{\wh{\mbb{X}}^{N_\ell}\times\wh{\mbb{U}}^{N_\ell}}\wh{c}(\mu,\theta)Q^{N_\ell}_t(d\mu\times d\theta)\notag\\
  &\geq\sum_{t=0}^{T-1}\beta^t\int_{\mc{P}(\prod_{j\in V}\mbb{X}^j)\times\mc{P}(\prod_{j\in V}\mbb{X}^j\times\mbb{U}^j)}\wh{c}(\mu,\theta)\wt{Q}_t(d\mu\times d\theta)\label{eq150}\\
  &=\int_{\prod_{t=0}^{T-1}(\mc{P}(\prod_{j\in V}\mbb{X}^j)\times\mc{P}(\prod_{j\in V}\mbb{X}^j\times\mbb{U}^j))}\sum_{t=0}^{T-1}\beta^t\wh{c}(\mu_t,\theta_t)\wt{Q}(d\mu_{0:T-1}\times d\theta_{0:T-1})\label{eq151}
\end{align}
where (\ref{eq150}) holds under Assumption \ref{ass1} since the stagewise costs
are jointly continuous and bounded, and we define $\wt{Q}$ to be the joint
strategic measure over the process of directing measures for the C-Ex limiting
ensemble. We emphasize that, as per the above computations, the cost depends on
the ensemble strategic measure only through that of the empirical distribution
process, and these converge along the chosen subsequence. Now, using Lemma
\ref{lem8}, repeating an identical sequence of steps as from
(\ref{eq152}-\ref{eq151}) establishes the convergence
\begin{align*}
  \limsup_{N\rightarrow\infty}J^N_T(\wt{P}\big|_N,\nu_0)&=\lim_{N\rightarrow\infty}\int_{\prod_{t=0}^{T-1}\ess\times\eas}\sum_{t=0}^{T-1}\beta^tc^N(x^{1:N}_t,u^{1:N}_t)\wt{P}\big|_N(dx^{1:N}_t\times du^{1:N}_t)\\
  &=\int_{\prod_{t=0}^{T-1}(\mc{P}(\prod_{j\in V}\mbb{X}^j)\times\mc{P}(\prod_{j\in V}\mbb{X}^j\times\mbb{U}^j))}\sum_{t=0}^{T-1}\beta^t\wh{c}(\mu_t,\theta_t)\wt{Q}(d\mu_{0:T-1}\times d\theta_{0:T-1}).
\end{align*} 
Consolidating our inequalities with this new expression for (\ref{eq151}), we find
\begin{align*}
  J^\ast_T(\nu_0)\geq\limsup_{N\rightarrow\infty}J^N_T(\wt{P}\big|_N,\nu_0)\geq\inf_{\pi\in\Pi^T_{\cex}}\limsup_{N\rightarrow\infty}J^N_T(\pi\big|_N,\nu_0)\geq J^\ast_T(\nu_0)
\end{align*}
where the second inequality follows by viewing $\wt{P}$ as being induced by a C-Ex policy, concluding the proof for $\mf{G}^\infty_T$. The case of $\mf{G}^\infty_\infty$ follows similarly, except in (\ref{eq150}) where the summation-integral interchange can be justified via montone convergence and the limit infimum is moved into the summation
via Fatou's lemma.\hfill{$\qed$}\\[5pt]
\indent A powerful consequence of Theorem \ref{thm7} is the applicability of the
De Finetti-type representation in Proposition \ref{prop4} to the infinite
ensemble without any loss of optimality. This justifies further
analysis/characterization of optima for $\mf{G}^\infty$ by considering a
mean-field MDP involving the directing measures governing a C-Ex infinite
ensemble. Accordingly, we now turn our attention to developing this mean-field MDP.
Afterward, we return to provide further structural results for $\mf{G}^\infty$
by studying how mean-field solutions can be taken to induce policies on the
infinite ensemble.

\subsection{Mean-Field MDP \texorpdfstring{$\wt{\mf{G}}$}{TEXT}}\label{sec10}
The De Finetti-type representation of C-Ex processes (Proposition \ref{prop4})
enables the development of a directing measure-valued `mean-field' MDP modeling
a controlled, but jointly C-Ex state-action profile for the infinite ensemble.
Moreover, our first main result, presented in Theorem \ref{thm7} establishes
that such a mean-field representation is applicable to $\mf{G}^\infty$ without
compromising operational performance of the ensemble. This will lead
to a characterization of the structure/informational demands of optima for
$\mf{G}^\infty$, along with a dynamic program for realizing them.\\[5pt]
\indent For this mean-field MDP, define the new state and action spaces
$\wt{\mbb{X}}:=\mc{P}(\prod_{j\in V}\mbb{X}^j)$,
$\wt{\mbb{U}}:=\mc{P}(\prod_{j\in V}\mbb{X}^j\times\mbb{U}^j)$, respectively.
Since the action is to represent a directing measure (in the sense of
proposition \ref{prop4}) over state-action pairs, we impose the fixed marginal
and factorizability action constraint given by
\begin{align}
  \wt{\mbb{U}}(\mu^{1:M}):=\bigg\{\theta\in\wt{\mbb{U}}\setsep\theta=\prod_{j\in V}\theta^j,\;\theta^j\in\mc{P}(\mbb{X}^j\times\mbb{U}^j),\;\theta^j(\cdot\times\mbb{U}^j)=\mu^j(\cdot)\;\forall\:j\in V\bigg\}.\label{eq137}
\end{align} 
\indent As for policies, let $\wt{\mbb{H}}_0:=\wt{\mbb{X}}$ and
$\wt{\mbb{H}}_t:=\wt{\mbb{H}}_{t-1}\times(\wt{\mbb{X}}\times\wt{\mbb{U}})$ for
$t\geq 1$ denote the (centralized) history spaces. Then an admissible policy at
time $t\geq 0$ is a stochastic kernel
$\tilde{\pi}_t\in\mc{P}(\wt{\mbb{U}}|\wt{\mbb{H}}_t)$ such that
$\tilde{\pi}_t(\wt{\mbb{U}}(\tilde{\mu}^{1:M}_t)|\wt{\mc{I}}_t)=1$ for all
$\wt{\mc{I}}_t=\{\tilde{\mu}^{1:M}_{0:t},\tilde{\theta}_{0:t-1}^{1:M}\}\in\wt{\mbb{H}}_t$.
Denote the set of these with $\wt{\Pi}_t$. Then the set of admissible policies
over the finite horizon is given by $\wt{\Pi}^T:=\prod_{t=0}^{T-1}\wt{\Pi}_t$,
and $\wt{\Pi}:=\prod_{t\geq 0}\wt{\Pi}_t$ for the infinite horizon.\\[5pt]
\indent In order to establish controlled, Markovian dynamics consistent with the original ensemble, and
ensure that the states remain factorizable (and hence a well-defined directing
measure in the sense of Definition \ref{def4}, compatible with the constraint
(\ref{eq137})) we have the following proposition.
\begin{proposition}\label{prop5} Suppose $(x^{1:\infty}_t,u^{1:\infty}_t)$ is
  C-Ex with state and state-action directing measures $\tilde{\mu}^{1:M}_t$,
  $\tilde{\theta}^{1:M}_t$ respectively $\forall\:t\geq 0$. Then
  $\forall\:A\in\mc{P}(\prod_{j\in V}\mbb{X}^j)$,
  \begin{align*}
    P(\tilde{\mu}^{1:M}_{t+1}\in A|\tilde{\mu}^{1:M}_{0:t},\tilde{\theta}^{1:M}_{0:t})=P(\tilde{\mu}^{1:M}_{t+1}\in A|\tilde{\mu}^{1:M}_t,\tilde{\theta}^{1:M}_t),\quad\forall\:t\geq 0.
  \end{align*}
  Moreover, these dynamics are deterministic such that
  $\tilde{\mu}^{1:M}_{t+1}=\wt{F}(\tilde{\mu}^{1:M}_t,\tilde{\theta}^{1:M}_t)$,
  where $\wt{F}=(\wt{F}_1,\wt{F}_2,\dots,\wt{F}_M)$ and
  $\wt{F}_j:\mc{P}(\prod_{\ell\in
  V}\mbb{X}^\ell)\times\mc{P}(\mbb{X}^j\times\mbb{U}^j)\rightarrow\mc{P}(\mbb{X}^j)$
  for all $t\geq 0$.
\end{proposition}
\noindent\textbf{Proof}. Let $j\in V$, $i\in\mc{C}^\infty_j$ and $A\in\mc{B}(\mbb{X}^j)$. Then,
\begin{align}
  P(x^i_{t+1}\in A|\tilde{\mu}^{1:M}_{0:t},\tilde{\theta}^{1:M}_{0:t})&=\int_{(\mbb{X}^j\times\mbb{U}^j)^t}P(x^i_{t+1}\in A|x,u,\tilde{\mu}^{1:M}_{0:t},\tilde{\theta}^{1:M}_{0:t})\mc{L}(x^i_{0:t},u^i_{0:t}|\tilde{\mu}^{1:M}_{0:t},\tilde{\theta}^{1:M}_{0:t})(dx\times du)\notag\\
  &=\int_{\mbb{X}^j\times\mbb{U}^j}\mc{T}_j(A|x,u,\tilde{\mu}^{1:M}_t)\mc{L}(x^i_t,u^i_t|\tilde{\mu}^{1:M}_{0:t},\tilde{\theta}^{1:M}_{0:t})(dx\times du)\label{eq145}\\
  &=\int_{\mbb{X}^j\times\mbb{U}^j}\mc{T}_j(A|x,u,\tilde{\mu}^{1:M}_t)\theta^j_t(dx\times du)\label{eq146}\\
  &=:\wt{F}_j(\tilde{\mu}^{1:M}_t,\tilde{\theta}^{j}_t)(A).\label{eq139}
\end{align}
In particular, (\ref{eq145}) follows the controlled Markov property for
individual agents, and that $\tilde{\mu}^{1:M}_t$ is the $\as$-limit of the
state empirical distributions by Lemma \ref{lem8}, and (\ref{eq146}) follows
since $\tilde{\theta}^j_t$ is the directing measure of $(x^i_t,u^i_t)$, implying
they are conditionally independent of
$(\tilde{\mu}^{1:M}_{0:t-1},\tilde{\theta}^{1:M}_{0:t-1})$ given
$\tilde{\theta}^j_t$ (see \cite[lemma
2.12]{Aldous_Ibragimov_1985}). By the chain rule, this implies
\begin{align*}
  \wt{F}_j(\tilde{\mu}^{1:M}_t,\tilde{\theta}^j_t)(A)=\int_{\mc{P}(\mbb{X}^j)}\tilde{\mu}^j_{t+1}(A)P(d\tilde{\mu}^j_{t+1}|\tilde{\mu}^{1:M}_{0:t},\tilde{\theta}^{1:M}_{0:t})
\end{align*}
so the conditional law on the right hand side must be a Dirac mass at the point $\wt{F}_j(\tilde{\mu}^{1:M}_t,\tilde{\theta}^j_t)$.\hfill{$\qed$}\\[5pt]
\indent By Proposition \ref{prop5}, the state directing measures admit controlled Markovian dynamics
such that
$\tilde{\mu}^j_{t+1}=\wt{F}_j(\tilde{\mu}^{1:M}_t,\tilde{\theta}^{1:M}_t)$
$\forall\:t\geq 0$, whence we set the initialization
$\tilde{\mu}^{1:M}_0=\nu_0$. Accordingly, the state-action profile for the
mean-field $\{(\tilde{\mu}^{1:M}_t,\tilde{\theta}^{1:M}_t)\}_{t\geq 0}$ defines
a MDP evolving under the joint state transition function
$\wt{F}:=(\wt{F}_1,\wt{F}_2,\dots,\wt{F}_M)$.\\[5pt]
\indent As for the cost, we consider either the finite or infinite horizon
criteria given by
\begin{align}
  \wt{J}_T(\pi,\nu_0)&:=\E^\pi_{\nu_0}\bigg[\sum_{t=0}^{T-1}\beta^t\wh{c}(\tilde{\mu}^{1:M}_t,\tilde{\theta}^{1:M}_t)\bigg],\quad\forall\:\pi\in\wt{\Pi}^T,\\
  \wt{J}(\pi,\nu_0)&:=\E^\pi_{\nu_0}\bigg[\sum_{t\geq 0}\beta^t\wh{c}(\tilde{\mu}^{1:M}_t,\tilde{\theta}^{1:M}_t)\bigg],\quad\forall\:\pi\in\wt{\Pi}.
\end{align}
In particular, the stagewise cost
$\wh{c}:\wt{\mbb{X}}\times\wt{\mbb{U}}\rightarrow\mbb{R}_+$ used here is exactly
that used for the finite MDPs $\wh{\mf{G}}^N$. We refer to the resulting
mean-field MDP with $\wt{\mf{G}}_T$ for the finite horizon, $\wt{\mf{G}}_\infty$
for the infinite horizon, or with $\wt{\mf{G}}$ to refer to either
interchangeably. For $\wt{\mf{G}}_T$ we seek an optimal policy
$\pi^\ast\in\wt{\Pi}^T$ satisfying the global optimality condition
\begin{align}
  \wt{J}_T(\pi^\ast,\nu_0)=\wt{J}^\ast_T(\nu_0):=\inf_{\pi\in\wt{\Pi}^T}\wt{J}_T(\pi,\nu_0)
\end{align}
with an analogous condition defining globally optimal policies and the value
function $\wt{J}^\ast$ for $\wt{\mf{G}}_\infty$. For our purposes, we require
the existence of optima for $\wt{\mf{G}}$, which we establish via the measurable
selection conditions encapsulated in the following lemma.
\begin{lemma}\label{lem11} Under assumption \ref{ass1}, $\wt{\mbb{X}}$,
  $\wt{\mbb{U}}$ are compact,
  $\tilde{\mu}^{1:M}\mapsto\wt{\mbb{U}}(\tilde{\mu}^{1:M})$ is a compact-valued,
  upper semicontinuous multifunction, and both $\wt{F}$, $\wh{c}$ are jointly
  continuous, with $\wh{c}$ additionally bounded.
\end{lemma}
\noindent\textbf{Proof}. The compactness of $\wt{\mbb{X}}$ and $\wt{\mbb{U}}$
follow that of $\mbb{X}^j$, $\mbb{U}^j$ $\forall\:j\in V$, and the weak
continuity/boundedness of $\wh{c}$ was previously established for
$\wh{\mf{G}}^N$. To see the compactness of the action constaint, recall our
remarks preceding Lemma \ref{lem8} (or see \cite[lemma 1.1]{Parthasarathy_1967})
and recall that fixed marginals are preserved under weak convergence. Then
$\wt{\mbb{U}}(\tilde{\mu}^{1:M})$ is a closed subset of $\wt{\mbb{U}}$, and
hence compact.\\[5pt]
\indent As for upper semicontinuity, let $\{\tilde{\mu}^{1:M}_n\}_{n\geq
1}\subset\wt{\mbb{X}}$ be such that
$\tilde{\mu}^{1:M}_n\rightarrow\tilde{\mu}^{1:M}$ as $n\rightarrow\infty$, and
take $\{\tilde{\theta}^{1:M}_n\}_{n\geq 1}\subset\wt{\mbb{U}}$ such that
$\tilde{\theta}^{1:M}_n\in\wt{\mbb{U}}(\tilde{\mu}^{1:M}_n)$ for $n\geq 1$. Then
by compactness, extract a subsequence
$\{\tilde{\theta}^{1:M}_{n_\ell}\}_{\ell\geq 1}$ converging weakly as
$\ell\rightarrow\infty$ to a factorizable limit $\tilde{\theta}^{1:M}$. Since
$\tilde{\mu}^{1:M}_{n_\ell}\rightarrow\tilde{\mu}^{1:M}$ as
$\ell\rightarrow\infty$, we get
$\tilde{\theta}^{1:M}\in\wt{\mbb{U}}(\tilde{\mu}^{1:M})$.\\[5pt]
\indent All that remains is the joint continuity of $\wt{F}$. For this, let
$\{(\tilde{\mu}^{1:M}_n,\tilde{\theta}^{1:M}_n)\}_{n\geq
1}\subset\wt{\mbb{X}}\times\wt{\mbb{U}}$ be such that
$(\tilde{\mu}^{1:M}_n,\tilde{\theta}^{1:M}_n)\rightarrow(\tilde{\mu}^{1:M},\tilde{\theta}^{1:M})$
weakly as $n\rightarrow\infty$. Then, for $j\in V$ and $g\in C_b(\mbb{X}^j)$
fixed we have
\begin{align}
  \int_{\mbb{X}^j}g(x)\wt{F}_j(\tilde{\mu}^{1:M}_n,\tilde{\theta}^{1:M}_n)(dx)&=\int_{\mbb{X}^j\times\mbb{U}^j}\bigg(\int_{\mbb{X}^j}g(x)\mc{T}_j(dx|x^\prime,u^\prime,\tilde{\mu}^{1:M}_n)\bigg)\tilde{\theta}^j_n(dx^\prime\times du^\prime)\notag
\end{align}
by Fubini's theorem. Letting
$h_n(x^\prime,u^\prime):=\int_\mbb{X}^jg(x)\mc{T}_j(dx|x^\prime,u^\prime,\tilde{\mu}^{1:M}_n)$
for $n\geq 1$, the joint continuity of $\mc{T}_j$ afforded by assumption
\ref{ass1} yields the continuous convergence of $\{h_n\}_{n\geq 1}$, whence an
application of the generalized dominated convergence theorem \cite{Serfozo_1982}
produces the convergence
$\wt{F}_j(\tilde{\mu}^{1:M}_n,\tilde{\theta}^{1:M}_n)\rightarrow\wt{F}_j(\tilde{\mu}^{1:M},\tilde{\theta}^{1:M})$
weakly as $n\rightarrow\infty$ $\forall\:j\in V$.\hfill{$\qed$}\\[5pt]
\indent As a consequence of Lemma \ref{lem11}, $\wt{\mf{G}}$ are amenable to the
classical theory of fully-observed MDPs. In particular, the dynamic programming
recursions for $\wt{\mf{G}}_T$ given by
\begin{align}
  &\wt{J}_{T-1,T}(\tilde{\mu}^{1:M}):=\inf_{\tilde{\theta}^{1:M}\in\wt{\mbb{U}}(\tilde{\mu}^{1:M})}\big\{\wh{c}(\tilde{\mu}^{1:M},\tilde{\theta}^{1:M})\big\},\label{eq140}\\
  &\wt{J}_{t,T}(\tilde{\mu}^{1:M}):=\inf_{\tilde{\theta}^{1:M}\in\wt{\mbb{U}}(\tilde{\mu}^{1:M})}\big\{\wh{c}(\tilde{\mu}^{1:M},\tilde{\theta}^{1:M})+\beta\wt{J}_{t+1,T}(\wt{F}(\tilde{\mu}^{1:M},\tilde{\theta}^{1:M}))\big\},\quad 0\leq t<T-1\label{eq141}
\end{align} 
are well-defined, and lead to an optimal deterministic, Markovian policy.
Further, $\wt{\mf{G}}_\infty$ admits an optimal deterministic, stationary
policy. In the next section, we consider how the controlled evolution of
directing measures can be equivalently realized by decentralized, C-Sym and
conditionally independent policies for the infinite ensemble.

\subsection{Structure of Optima for \texorpdfstring{$\mf{G}^\infty$}{TEXT}: Part II}\label{sec11}
In this section we continue our characterization of the structure of optimal
policies for $\mf{G}^\infty$, establishing the optimality of C-Sym,
conditionally independent policies under the dCMF IS for $\mf{G}^\infty$. For
this we first develop a cost-preserving map from policies for the mean-field
representation $\wt{\mf{G}}$ to the infinite ensemble, which then gives way to a
refined version of Theorem \ref{thm8}. First, we restate the definition provided
for the dCMF IS with the more precise notions afforded by the preceding
discussion on C-Ex processes.
\begin{definition}[Decentralized Cluster Mean-Field Sharing Information
  Structure]\label{def6} Suppose $(x^{1:\infty}_t,u^{1:\infty}_t)$ is C-Ex with
  state directing measure $\tilde{\mu}^{1:M}_t$ for all $t\geq 0$. Then the
  decentralized cluster mean-field sharing information structure (dCMF IS) is
  given by
  \begin{align*}
    \dcmfis:=\{x^i_{0:t},u^i_{0:t-1},\tilde{\mu}^{1:M}_{0:t}\},\quad\forall\:i\in\mbb{N},t\geq 0.
  \end{align*}
\end{definition}
Under the dCMF IS, the common information available to all agents is at most the
state directing-measure profile $\tilde{\mu}^{1:M}_{0:t}$, but a Markovian
policy associated to agent $i\in\mbb{N}$ under this IS would depend on only
$\{x^i_t,\tilde{\mu}^{1:M}_t\}$ at time $t\geq 0$. The following result
establishes that deterministic policies for $\wt{\mf{G}}$ induce C-Sym,
conditionally independent policies for $\mf{G}^\infty$ which are Markovian and
depend on the dCMF IS, but does not address their operational performance.
\begin{theorem}\label{thm9} A deterministic, Markovian policy
  $\tilde{\gamma}^{1:M}\in\wt{\Pi}^T$ for $\wt{\mf{G}}_T$
  ($\tilde{\gamma}^{1:M}\in\wt{\Pi}$ for $\wt{\mf{G}}_\infty$) induces a
  decentralized (under the dCMF IS) policy $\tilde{\pi}\in\Pi^{T}_{\csym}$ for
  $\mf{G}^\infty_T$ ($\tilde{\pi}\in\Pi_{\csym}$ for $\mf{G}^\infty_\infty$)
  such that
  \begin{align*}
    \tilde{\pi}_t(du^{1:\infty}_t|x^{1:\infty}_{0:t},u^{1:\infty}_{0:t-1})=\prod_{j\in V}\prod_{i\in\mc{C}^\infty_j}\tilde{\pi}^j_t(du^i_t|x^i_t,\tilde{\mu}^{1:M}_t),\quad\forall\:t\geq 0
  \end{align*}
  where $\tilde{\pi}^j_t\in\mc{P}(\mbb{U}^j|\mbb{H}_t)$ satisfies
  $\tilde{\gamma}^j_t(\tilde{\mu}^{1:M}_t)=\tilde{\pi}^j_t\otimes\tilde{\mu}^{j}_t$
  $\forall\:j\in V$.
\end{theorem}
\noindent\textbf{Proof}. Consider $\mf{G}^\infty_T$. For $0\leq t<T$, write
$\tilde{\theta}^{1:M}_t=\tilde{\gamma}^{1:M}(\tilde{\mu}^{1:M}_t)$. Then
conditioning on $\tilde{\mu}^{1:M}_t$, the map $\tilde{\gamma}^j\mapsto\pi^j_t$
specified by the disintegration
$\tilde{\theta}^j_t=\pi^j_t\otimes\tilde{\mu}^j_t$ is measurable, with
$\pi^j_t\in\mc{P}(\mbb{U}^j|\mbb{X}^j)$ $\forall\:j\in V$. Hence, we may
construct the induced policy
$\tilde{\pi}=\tilde{\pi}^{1:\infty}\in\Pi^T_{\csym}$ so that for any $t\geq 0$,
$j\in V$, and $i\in\mc{C}^\infty_j$, $\tilde{\pi}^i_t$ denotes the composition
$(x^i_t,\tilde{\gamma}^{1:M}_t,\tilde{\mu}^{1:M}_t)\mapsto
(x^i_t,\pi^j_t)\mapsto \pi^j_t(\cdot|x^i_t)$.\\[5pt]
\indent All that remains to show is that, under $\tilde{\pi}$, the resulting
directing measures for $\mf{G}^\infty_T$ coincide with the directing measure
process induced by $\tilde{\gamma}^{1:M}$ for $\wt{\mf{G}}_T$.  Note first that
$\nu_0=\tilde{\mu}^{1:M}_0$ is the directing measure of $x^{1:\infty}_0$ by
assumption. To proceed by induction, suppose $\tilde{\mu}^{1:M}_t$ is the
directing measure of $x^{1:\infty}_t$ for some $0\leq t<T$. Then
$\mc{L}(x^i_t,u^i_t|\tilde{\mu}^{1:M}_t)=\pi^{C(i)}_t\otimes\tilde{\mu}^{C(i)}=\tilde{\theta}^{C(i)}_t$
for $i\in\mbb{N}$. Hence, $\forall\:j\in V$, $i\in\mc{C}^\infty_j$ and
$A\in\mc{B}(\mbb{X}^j)$ we have
\begin{align*}
  P(x^i_{t+1}\in A|\tilde{\mu}^{1:M}_t)=\int_{\mbb{X}^j\times\mbb{U}^j}P(x^i_{t+1}\in A,dx^i_t,du^i_t|\tilde{\mu}^{1:M}_t)=\int_{\mbb{X}^j\times\mbb{U}^j}\mc{T}_j(A|x^i,u^i,\tilde{\mu}^{1:M}_t)\tilde{\theta}^j_t(dx^i\times du^i)
\end{align*}
so now $\mc{L}(x^i_{t+1}|\tilde{\mu}^{1:M}_t)=\tilde{\mu}^j_{t+1}$ under
$\tilde{\pi}$. However, since $\tilde{\mu}^{1:M}_0=\nu_0$ is deterministic, this
resolves to $\mc{L}(x^i_{t+1})=\tilde{\mu}^j_{t+1}$, which holds $\forall\:t\geq
0$ by induction, and the IS used by $\tilde{\pi}$ is the dCMF IS. Clearly, the
construction of $\tilde{\pi}$ and induction argument would work to establish the
same result for $\mf{G}^\infty_\infty$.\hfill{$\qed$}\\[5pt]
\indent By the proof of Theorem \ref{thm9}, the time dependence of coordinate
policies $\tilde{\pi}_t$ is inherited from $\tilde{\gamma}^{1:M}_t$, and as such
$\tilde{\pi}$ is stationary whenever $\tilde{\gamma}^{1:M}$ is. Furthermore,
Theorem \ref{thm9} implies that deterministic policies for $\wt{\mf{G}}$ induce
propagation of chaos for the infinite ensemble; if initial states are $\iid$
within clusters (only independent between) then so are all subsequent states
under induced policy $\tilde{\pi}$. Moreover, the trajectory of cluster
distributions is deterministic. Beyond merely recasting centralized mean-field
policies as decentralized ones for the ensemble, this procedure turns out to be
cost preserving between the two problems.
\begin{theorem}\label{thm10} Let $\tilde{\pi}$ be the decentralized policy
  induced by $\tilde{\gamma}^{1:M}$ as in Theorem \ref{thm9}. Then
  \begin{align*}
    J_T(\tilde{\pi},\nu_0)=\wt{J}_T(\tilde{\gamma}^{1:M},\nu_0),\quad\text{and}\quad J(\tilde{\pi},\nu_0)=\wt{J}(\tilde{\gamma}^{1:M},\nu_0)
  \end{align*}
  with the provisio that $\tilde{\pi}\in\Pi^T_{\csym}$,
  $\tilde{\gamma}^{1:M}\in\wt{\Pi}^T$ in the former and
  $\tilde{\pi}\in\Pi_{\csym}$, $\tilde{\gamma}^{1:M}\in\wt{\Pi}$ in the latter.
\end{theorem}
\noindent\textbf{Proof}. Consider $\mf{G}^\infty_T$. Let
$Q\in\mc{P}(\prod_{t=0}^{T-1}\bar{\mbb{X}}\times\bar{\mbb{U}})$ be the joint
strategic measure induced by $\tilde{\pi}$. Then the sequence of costs  induced
by $Q$ leading to $J_T(\tilde{\pi},\nu_0)$ are given by
\begin{align}
  J^N_T(\tilde{\pi},\nu_0)&=\int_{\prod_{t=0}^{T-1}\bar{\mbb{X}}\times\bar{\mbb{U}}}\sum_{t=0}^{T-1}\beta^tc^N(x^{1:N}_t,u^{1:N}_t)Q(dx^{1:N}_{0:T-1}\times du^{1:N}_{0:T-1})\notag\\
  &=\int_{\prod_{t=0}^{T-1}\ess\times\eas}\sum_{t=0}^{T-1}\beta^t\wh{c}(\mu^{1:M}[x^{1:N}_t],\mu^{1:M}[(x^{1:N}_t,u^{1:N}_t)])Q(dx^{1:N}_{0:T-1}\times du^{1:N}_{0:T-1}).\label{eq138}
\end{align}
Let $\xi^{1:M}$ denote the directing measure of
$(x^{1:\infty}_{0:T-1},u^{1:\infty}_{0:T-1})$ under $\tilde{\pi}$, with
$\Xi:=\mc{L}(\xi^{1:M})$. Then, with $\mu^{1:M}_{t,N}:=\mu^{1:M}[x^{1:N}_t]$,
$\theta^{1:M}_{t,N}:=\mu^{1:M}[(x^{1:N}_t,u^{1:N}_t)]$ and Fubini's theorem,
rewrite (\ref{eq138}) to obtain
\begin{align*}
  J^N_T(\tilde{\pi},\nu_0)&=\int_{\mc{P}(\prod_{t=0}^{T-1}\prod_{j\in V}\mbb{X}^j\times\mbb{U}^j)}\bigg(\int_{\prod_{t=0}^{T-1}\ess\times\eas}\sum_{t=0}^{T-1}\beta^t\wh{c}(\mu^{1:M}_{t,N},\theta^{1:M}_{t,N})\cdots\\
  &\qquad\qquad\qquad\qquad\cdots\prod_{j\in V}\prod_{i\in\mc{C}^N_j}\xi^{j}(dx^{i}_{0:T-1}\times du^{i}_{0:T-1})\bigg)\Xi(d\xi^{1:M}).
\end{align*}
Now, since $\xi^{1:M}=\tilde{\theta}^{1:M}_{0:T-1}$, the proof of Theorem
\ref{thm9} implies $\Xi$ is a Dirac mass on the (deterministic) trajectory of
$\tilde{\theta}^{1:M}_{0:T-1}$ where
$\tilde{\theta}^{1:M}_t=\tilde{\gamma}^{1:M}_t(\tilde{\mu}^{1:M}_t)$ and
$\tilde{\mu}^{1:M}_{0:T-1}$ follow dynamics (\ref{eq139}). Hence, invoking the
dominated convergence theorem followed by Lemma \ref{lem8}, take the limit as
$N\rightarrow\infty$ to find
\begin{align*}
  \lim_{N\rightarrow\infty}J^N_T(\tilde{\pi},\nu_0)&=\int_{\mc{P}(\prod_{t=0}^{T-1}\prod_{j\in V}\mbb{X}^j\times\mbb{U}^j)}\bigg(\int_{\prod_{t=0}^{T-1}\ess\times\eas}\sum_{t=0}^{T-1}\beta^t\wh{c}(\tilde{\mu}^{1:M}_{t},\tilde{\theta}^{1:M}_{t,N})\cdots\\
  &\qquad\qquad\qquad\qquad\cdots\prod_{j\in V}\prod_{i\in\mc{C}^N_j}\xi^{j}(dx^{i}_{0:T-1}\times du^{i}_{0:T-1})\bigg)\Xi(d\xi^{1:M})\\
  &=\sum_{t=0}^{T-1}\beta^t\wh{c}(\tilde{\mu}^{1:M}_t,\tilde{\theta}^{1:M}_t)\\
  &=\wt{J}_T(\tilde{\gamma}^{1:M},\nu_0)
\end{align*}
which yields the relation
$J_T(\tilde{\pi},\nu_0)=\wt{J}_T(\tilde{\gamma}^{1:M},\nu_0)$. As previously, we
note that all arguments used here remain valid when the horizon is
infinite.\hfill{$\qed$}\\[5pt]
\indent Finally, we revisit and refine Theorem \ref{thm7} in light of our
findings for Theorems \ref{thm9} and \ref{thm10}.
\begin{theorem}\label{thm11} Let $\tilde{\pi}^\ast$ be a decentralized policy
  (under the dCMF IS) for $\mf{G}^\infty$ induced by an optimal, deterministic
  policy for $\wt{\mf{G}}$ as in Theorem \ref{thm9}. Then
  \begin{align*}
    &\inf_{\pi\in\bar{\Pi}^T}J_T(\pi,\nu_0)=\inf_{\pi\in\Pi^T_{\csym}}J_T(\pi,\nu_0)=J_T(\tilde{\pi}^\ast,\nu_0),\quad\text{and}\\
    &\inf_{\pi\in\bar{\Pi}^T}J(\pi,\nu_0)=\inf_{\pi\in\Pi_{\csym}}J(\pi,\nu_0)=J(\tilde{\pi}^\ast,\nu_0).
  \end{align*} 
\end{theorem}
\noindent\textbf{Proof}. Consider $\mf{G}^\infty_T$. Recall the lower bound in
(\ref{eq135}), which yields
\begin{align*}
  J^\ast_T(\nu_0)\geq\liminf_{\ell\rightarrow\infty}J^{N_\ell}_T(\pi^{N_\ell,\ast},\nu_0)
\end{align*}
for a subsequence $\{N_\ell\}_{\ell\geq 1}$, and with
$\pi^{N_\ell,\ast}\in\Pi^{N_\ell,T}_{\cex}$ optimal for $\mf{G}^{N_\ell}_T$ for
all $\ell\geq 1$. Along this subsequence, the strategic measures converge (along
any finite marginal) to one which is C-Ex such that
$(x^{1:\infty}_t,u^{1:\infty}_t)$ is C-Ex, say with directing measure
$\tilde{\theta}^{1:M}_t$, and state marginal $\tilde{\mu}^{1:M}_t$ for all
$t\geq 0$. By precisely the same inductive argument as in the proof of Theorem
\ref{thm9}, we find that $\tilde{\mu}^{1:M}_{0:T-1}$ is a deterministic flow
evolving under dynamics $\wt{F}$, such that the directing measures
$\tilde{\theta}^{1:M}_{0:T-1}$ can be realized by an admissible, deterministic
policy for $\wt{\mf{G}}_T$. Thus, if $\tilde{\gamma}^{1:M,\ast}$ is optimal for
$\wt{\mf{G}}_T$ and induces a decentralized policy $\tilde{\pi}^\ast$ for
$\mf{G}^\infty_T$ as in Theorem \ref{thm9}, then by Theorem \ref{thm10}, we
obtain the sequence of inequalities
\begin{align}
  J^\ast_T(\nu_0)\geq\inf_{\tilde{\pi}\in\wt{\Pi}^T}\wt{J}_T(\tilde{\pi},\nu_0)=\wt{J}(\tilde{\gamma}^{1:M,\ast},\nu_0)=J_T(\tilde{\pi}^\ast,\nu_0)\geq\inf_{\pi\in\Pi^T_{\csym}}J_T(\pi,\nu_0)\geq J^\ast_T(\nu_0)\label{eq142}
\end{align}
which establishes the result for $\mf{G}^\infty_T$. As previously, none of the
above arguments were sensitive to the finiteness of the time horizon, and so may
be repeated to obtain an analogous result for
$\mf{G}^\infty_\infty$.\hfill{$\qed$}\\[5pt]
\indent Theorem \ref{thm11} acts as a verification result for $\mf{G}^\infty$
and produces two important corollaries. On verification, it certifies that any
decentralized, C-Sym policy induced (in the sense of Theorem \ref{thm9}) by a
solution to the dynamic programming recursions (\ref{eq140}-\ref{eq141}) (or a
corresponding discounted cost optimality equation) for $\wt{\mf{G}}$ is optimal
for $\mf{G}^\infty$. Complementing the computational utility of this mean-field
dynamic program, the first corollary demonstrates the C-Sym truncations
$\tilde{\pi}^\ast|_N$ are asymptotically optimal for $\mf{G}^N$ (in the large
population limit) while retaining the same decentralization and structural
properties as a C-Sym optimum for $\mf{G}^\infty$. We emphasize that this
approximation result is valid within the class of completely centralized
policies for the finite ensemble. A second, more theoretically-oriented
corollary stems from the observation that, under decentralized, C-Sym policies
the directing measures for $\mf{G}^\infty$ admit a deterministic flow,
upgradining $\tilde{\mu}^{1:M}_t$ from an array of conditional cluster
distributions to an array of (unconditional) laws over each cluster. As noted
above, this can be interpreted as a probabilistic propagation of chaos --
producing a connection with the McKean-Vlasov-type formalisms commonly deployed
to analyze non-cooperative models. We conclude our study with a final section
detailing these corollaries.

\section{Optimality of the McKean-Vlasov Team Representation, and \texorpdfstring{$\varepsilon$}{TEXT}-Optimality of Cluster-Symmetric, Decentralized Policies for \texorpdfstring{$\mf{G}^N$}{TEXT}}\label{sec12}
This section deals with two important consequences of our analyses in the
preceding section. First is the verification that a decentralized, McKean-Vlasov
team problem with a system of $M$ cluster-representatives is equivalent to
$\mf{G}^\infty$ and leads to an optimal policy. The second deals with the
approximate optimality of C-Sym, conditionally independent and decentralized
policies (utilizing only the dCMF IS) for $\mf{G}^N$ for $N$ sufficiently large.
Presenting these in turn, we begin by formulating the McKean-Vlasov team in
analogy with $\mf{G}^\infty$.\\[5pt]
\indent Consider a collection of $M$ cluster-representative agents, with state
dynamics given by
\begin{align}
  x^{\rep_j}_{t+1}=f_j(x^{\rep_j}_t,u^{\rep_j}_t,\mu^{\rep_1}_t,\mu^{\rep_2}_t,\dots,\mu^{\rep_M}_t,w^{\rep_j}_t)=:f_j(x^{\rep_j}_t,u^{\rep_j}_t,\mu^{\rep_{1:M}}_t,w^{\rep_j}_t),\quad j\in V,\;t\geq 0
\end{align}
and where the functions $f_j$, disturbances
$w^{\rep_j}_t\overset{\iid}{\sim}\nu^j$ and initializations
$x^{\rep_j}_0\sim\nu_0^j$ are the same as for $\mf{G}^N$, but now with the
consistency condition $\mu^{\rep_j}_t:=\mc{L}(x^{\rep_j}_t)$ so that
$\mu^{\rep_j}_0=\nu_0^j$ $\forall\:j\in V$. Policy spaces are defined in analogy
with $\Pi^{M,T}_{\csym}$ for $\mf{G}^N$ too (i.e., with only one agent per
cluster) except with the decentralized IS such that representative $j\in V$ at
time $t\geq 0$ has access to
\begin{align*}
  \mc{I}^{\rep_j}_t:=\{x^{\rep_j}_{t},\mu^{\rep_{1:M}}_{t}\}.
\end{align*}
Denoting the joint policy space with $\Pi^{\rep,T}$ for the finite horizon, and
$\Pi^{\rep}$ for the infinite horizon, we seek global optima
$\inf_{\pi\in\Pi^{\rep,T}}J^{\rep}_T(\pi,\nu_0)$ or
$\inf_{\pi\in\Pi^{\rep}}J^{\rep}(\pi,\nu_0)$ where the cost functions are given
by
\begin{align*}
  &J^{\rep}_T(\pi,\nu_0):=\E^{\pi}_{\nu_0}\bigg[\sum_{t=0}^{T-1}\beta^t\sum_{j\in V}c^j(x^{\rep_j}_t,u^{\rep_j}_t,\mu^{\rep_j}_t)\bigg],\quad\text{or}\\
  &J^{\rep}(\pi,\nu_0):=\E^{\pi}_{\nu_0}\bigg[\sum_{t=0}^{\infty}\beta^t\sum_{j\in V}c^j(x^{\rep_j}_t,u^{\rep_j}_t,\mu^{\rep_j}_t)\bigg]
\end{align*}
for the finite and infinite horizon problems, respectively. The above
constitutes a (decentralized) McKean-Vlasov team problem with $M$ representative
agents, which we denote with $\mf{G}^{\rep}$ without specifying the cost
criterion. As formulated, the IS for $\mf{G}^{\rep}$ coincides with the dCMF IS
for $\mf{G}^\infty$ with the caveat that representative agents share their state
laws $\{\mu^{\rep_j}_t\}_{j\in V}$ rather than directing measures
$\{\tilde{\mu}^{1:M}_t\}_{j\in V}$. However, as was seen in the proof of Theorem
\ref{thm9}, this distinction is moot for $\mf{G}^\infty$, since the flow of
directing measures is deterministic. From this follows the first main corollary
of theorem \ref{thm11}.
\begin{corollary}\label{cor2} There exists a cost-preserving bijection between
  $\Pi^{\rep,T}$ ($\Pi^{\rep}$) and decentralized policies in $\Pi^{T}_{\csym}$
  ($\Pi_{\csym}$) under the dCMF IS.
\end{corollary}
\noindent\textbf{Proof}. Consider a policy $\pi^{\rep_{1:M}}$ for
 $\mf{G}^{\rep}$, and the bijection given by
 $\tilde{\pi}^i_t=\pi^{\rep_{C(i)}}_t$ for all $i\in\mbb{N}$, $t\geq 0$. Then
 the resulting joint policy $\tilde{\pi}$ for $\mf{G}^\infty$ is decentralized
 and C-Sym, so $(x^{1:\infty}_t,u^{1:\infty}_t)$ is C-Ex $\forall\:t\geq 0$. As
 was seen by the induction argument in the proof of Theorem \ref{thm9}, the
 resulting directing measures
 $\{(\tilde{\mu}^{1:M}_t,\tilde{\theta}^{1:M}_t)\}_{t\geq 0}$ admit the
 deterministic flow given in (\ref{eq139}). Additionally, since both
 $\tilde{\mu}^{1:M}_0=\nu_0=\mu^{\rep_{1:M}}_0$ and
 $\tilde{\theta}^{1:M}_0=\pi^{\rep_{1:M}}_0\otimes\tilde{\mu}^{1:M}_0$ is
 deterministic, we have
 $\mc{L}(x^{\rep_{C(i)}}_0,u^{\rep_{C(i)}}_0)=\tilde{\theta}^{C(i)}_0=\mc{L}(x^i_0,u^i_0)$
 $\forall\:i\in\mbb{N}$. Combining these, $\forall\:j\in V$ we have
\begin{align}
  \tilde{\mu}^j_1(A)=\wt{F}_j(\tilde{\mu}^{1:M}_0,\tilde{\theta}^{1:M}_0)(A)&=\int_{\mbb{X}^j\times\mbb{U}^j}\mc{T}_j(A|x,u,\tilde{\mu}^{1:M}_0)\tilde{\theta}^j_0(dx\times du)\notag\\
  &=\int_{\mbb{X}^j\times\mbb{U}^j}\mc{T}_j(A|x,u,\mu^{\rep_{1:M}}_0)\mc{L}(x^{\rep_j}_0,u^{\rep_j}_0)(dx\times du)\notag\\
  &=\mc{L}(x^{\rep_j}_1|\mu^{\rep_{1:M}}_0)(A)\label{eq143}
\end{align}
$\forall\:A\in\mc{B}(\mbb{X}^j)$. Now, since $\mu^{\rep_{1:M}}_0$ is
deterministic, the above implies
$\tilde{\mu}^j_1=\mc{L}(x^{\rep_j}_1)=\mu^{\rep_j}_1$ $\forall\:j\in V$. Thus,
by repeating the above arguments (i.e.,
$\tilde{\theta}^{1:M}_1=\pi^{\rep_{1:M}}_1\otimes\tilde{\mu}^{1:M}_1$ is
deterministic, and so
$\mc{L}(x^{\rep_{C(i)}}_1,u^{\rep_{C(i)}}_1)=\mc{L}(x^i_1,u^i_1)$
$\forall\:i\in\mbb{N}$ and so the computations in (\ref{eq142}) yield
$\tilde{\mu}^{1:M}_2=\mu^{\rep_{1:M}}_2$) we can recursively show
$\tilde{\mu}^{j}_t=\mc{L}(x^{\rep_j}_t)$ and
$\tilde{\theta}^{j}_t=\mc{L}(x^{\rep_j}_t,u^{\rep_j}_t)$ for all $t\geq 0$.
Consequently, for the finite horizon,
\begin{align*}
  J^{\rep}_T(\pi^{\rep_{1:M}},\nu_0)&=\sum_{t=0}^{T-1}\beta^t\sum_{j\in V}\int_{\mbb{X}^j\times\mbb{U}^j}c^j(x^{\rep_j}_t,u^{\rep_j}_t,\mu^{\rep_{1:M}}_t)\tilde{\theta}^j_t(dx^{\rep_j}_t\times du^{\rep_j}_t)\\
  &=\sum_{t=0}^{T-1}\beta^t\wh{c}(\tilde{\mu}^{1:M}_t,\tilde{\theta}^{1:M}_t)\\
  &=\wt{J}_T(\tilde{\pi},\nu_0)
\end{align*}
and the same arguments hold for the infinite horizon, owing to the discount
factor. If instead we had begun with a C-Sym, decentralized policy $\tilde{\pi}$
for $\mf{G}^\infty$, then the McKean-Vlasov policy would be
$\pi^{\rep_j}_t=\tilde{\pi}_t^{k_j}$, where $k_j$ is the sole element of
$\mc{C}^M_j$ $\forall\:j\in V$, $t\geq 0$, and the remaining arguments would be
precisely as above.\hfill{$\qed$}\\[5pt]
\indent As a consequence of the reasoning of corollary \ref{cor2}, the Bellman
recursions developed for $\wt{\mf{G}}$ actually coincide with the same for a
measure-valued representation of the decentralized McKean-Vlasov team. Moreover,
a policy satisfying these $\as$ can be freely converted between a solution for
$\mf{G}^{\rep}$ and a dCMF decentralized, C-Sym solution for
$\mf{G}^\infty$.\\[5pt]
\indent For our second concern on approximate solutions to $\mf{G}^N$, we will
show that truncating a dCMF decentralized, C-Sym solution for $\mf{G}^\infty$
yields a policy which is near-optimal within the class of admissible (even
centralized) policies provided $N\geq M$ is sufficiently large. 
\begin{corollary}\label{cor3} Let assumption \ref{ass1} hold. Then,
  $\exists\:\{\pi^{N,\ast}\}_{N\geq M}$ with $\pi^{N,\ast}\in\Pi^{N,T}_{\csym}$
  $\forall\:N\geq M$ decentralized under the dCMF IS such that
  $\forall\:\varepsilon>0$, $\exists\:K\geq M$ such that $N\geq K$ implies
  \begin{align*}
    J^{N,\ast}_T(\nu_0)\leq J^N_T(\pi^{N,\ast},\nu_0)\leq J^{N,\ast}_T(\nu_0)+\varepsilon.
  \end{align*}
\end{corollary}
\noindent {\bf Proof}. Consider the finite horizon setup. Careful review of
(\ref{eq135}) in the proof of Theorem \ref{thm7} and (\ref{eq142}) in the proof
of Theorem \ref{thm11} reveals that
$\lim_{N\rightarrow\infty}J^{N,\ast}_T(\nu_0)=J^\ast_T(\nu_0)$. Further, with
$\pi^\ast\in\Pi^T_{\csym}$ an optimal, dCMF decentralized policy for
$\mf{G}^\infty_T$, we have
$\lim_{N\rightarrow\infty}J^N_T(\pi^{N,\ast},\nu_0)=J^\ast_T(\nu_0)$, where
$\pi^{N,\ast}$ is the restriction of $\pi^\ast$ to the first $N\geq M$ agents.
Hence, fixing $\varepsilon>0$, $\exists\:K\geq M$ so that $N\geq K$ implies both
\begin{align*}
  \inf_{\pi\in\bar{\Pi}^{N,T}}J^N_T(\pi,\nu_0)\geq &J^\ast_T(\nu_0)-\frac{\varepsilon}{2},\quad\text{and}\\
  &J^\ast_T(\nu_0)\geq J^N_T(\pi^{N,\ast},\nu_0)-\frac{\varepsilon}{2}.
\end{align*}
Combining these, we produce the upper bound $J^N_T(\pi^{N,\ast},\nu_0)\leq
J^{N,\ast}_T(\nu_0)+\varepsilon$. The result follows. As has been the case
throughout, the infinite horizon case requires no
innovation.\hfill{$\qed$}\\[5pt]
\indent Thus, the asymptotically optimal decentralized policies are obtained as
a simple truncation of the solution for $\mf{G}^\infty$ (or $\mf{G}^{\rep}$). It
should be emphasized that the dCMF IS used by the truncated policies is the same
as for the mean-field; agents have access to the state directing measures
$\tilde{\mu}^{1:M}_{t}$, which are deterministic and satisfy
$\tilde{\mu}^j_t=\mc{L}(x^{\rep_j}_t)$ (as opposed to a finite-population
estimate via the empirical measures $\mu^{1:M}[x^{1:N}_t]$).

\section{Conclusion}\label{sec13}
In this paper, we provide a rigorous study of the structure and information
demands of optimal policies for stochastic teams comprised of a finite number of
a finite number of agent clusters in the finite and infinite population regimes.
We find that, without \textit{a priori} constraints on the policy structure or
IS, there exists a C-Sym joint policy, conditionally independent and fully
decentralized (under cluster mean-field information sharing) which is optimal
for $\mf{G}^\infty$, and asymptotically optimal for $\mf{G}^N$ as
$N\rightarrow\infty$. This policy is induced by one for a corresponding
mean-field MDP, and obtainable via a mean-field dynamic program. Moreover, this
mean-field MDP is equivalent to a decentralized McKean-Vlasov-type team of
cluster representative agents, which justifies the optimality of the latter as a
representation for the infinite ensemble $\mf{G}^\infty$. We expect
$\wt{\mf{G}}$ to serve as a platform for future developments on learning
algorithms, and suggest our model could serve as an intermediate aiding
structural results for graphon mean-field teams with unbounded agent diversity.
\appendix
\section{Measurable Selection Conditions for \texorpdfstring{$\wh{\mf{G}}^N$}{TEXT}}\label{app1}
First, let us recall the definition of upper semicontinuity for set-valued
functions (i.e., multifunctions).
\begin{definition}[Upper Semicontinuity of Multifunctions]\label{def7} Let
  $\mbb{Y}$, $\mbb{V}$ be standard Borel, and consider a map
  $y\mapsto\mbb{V}(y)$ such that $\mbb{V}(y)\subset\mbb{V}$
  $\forall\:y\in\mbb{Y}$. Then $\mbb{V}(\cdot)$ is called upper semicontinuous
  if the following holds $\forall\:y\in\mbb{Y}$: if $\{y_n\}_{n\geq
  1}\subset\mbb{Y}$ satisfies $y_n\rightarrow y $ as $n\rightarrow\infty$ and
  $v_n\in\mbb{V}(y_n)$ for all $n\geq 1$, then we can extract a convergent
  subsequence $\{v_{n_k}\}_{k\geq 1}$ so that
  $\lim_{k\rightarrow\infty}v_{n_k}\in\mbb{V}(y)$.
\end{definition}
Armed with this definition, the following theorem is a classical result
verifying the measurable selection hypothesis for the MDP $\wh{\mf{G}}^N$
(although, we invoke an analogous result for $\mf{G}^N$ and $\wt{\mf{G}}$). In
particular, the following is sufficient to guarantee the existence of minimizing
measurable selectors such that the corresponding Bellman recursions are
well-defined. For reference, see \cite[$\S$3.3, Theorem
3.3.5]{Hernandez-Lerma_Lasserre_1996}.
\begin{theorem}\label{thm2} Let assumption \ref{ass1}
  hold and consider $\wh{\mf{G}}^N$ for $N\geq M$ fixed. Then $\wh{\mbb{X}}^N$, $\wh{\mbb{U}}^N$ are compact,
  $\wh{c}:\mc{P}(\prod_{j\in V}\mbb{X}^j)\times\mc{P}(\prod_{j\in V}\mbb{X}^j\times\mbb{U}^j)\rightarrow\mbb{R}_+$ is jointly continuous and bounded,
  $x^{1:N}\mapsto\wh{\mbb{U}}^N(\mu^{1:M}[x^{1:N}])$ is a compact-valued, upper semicontinuous multifunction,
  and the transition kernel $\wh{\mc{T}}$ is weak Feller. 
\end{theorem}
\noindent {\bf Proof}. Compactness of the state and action spaces follows $\wh{\mbb{X}}^N=\mu^{1:M}[\ess]$, $\wh{\mbb{U}}^N=\mu^{1:M}[\ess\times\eas]$ since $\ess$ and $\eas$ are each compact and $\mu^{1:M}$ is a weakly continuous function.
For the weak Feller property for $\wh{\mc{T}}^N$, let $\{(\mu^{1:M}_{t,k},\theta^{1:M}_{t,k})\}_{k\geq 1}\subseteq\wh{\mbb{X}}^N\times\wh{\mbb{U}}^N$ be a sequence where each pair $(\mu^{1:M}_{t,k},\theta^{1:M}_{t,k})$ is induced by an
ensemble $(x^{1:N}_{t,k}, u^{1:N}_{t,k})\in\ess\times\eas$ and both $\theta^{1:M}_{t,k}\rightarrow\theta^{1:M}_t$ and $\mu^{1:M}_{t,k}\rightarrow\mu^{1:M}_{t}$ jointly
as $k\rightarrow\infty$.
In conjunction, consider the quotient space induced by the equivalence relation
\begin{align*}
  (x^{1:N},u^{1:N})\sim(\hat{x}^{1:N},\hat{u}^{1:N})\quad\Longleftrightarrow\quad \exists\,\sigma\in\mc{S}_{\mc{C}^N}\quad\text{s.t.}\quad (x^{\sigma(1):\sigma(N)},u^{\sigma(1):\sigma(N)})=(\hat{x}^{1:N},\hat{u}^{1:N}) 
\end{align*}
$\forall\:(x^{1:N},u^{1:N}),(\hat{x}^{1:N},\hat{u}^{1:N})\in\ess\times\eas$. That this
indeed forms an equivalence relation is evident from  the closedness of $\mc{S}_{\mc{C}^N}$ under composition and that any $\sigma\in\mc{S}_{\mc{C}^N}$
is a bijection on $\mc{N}$. For
$k\geq 1$, denote the representatives of $x^{1:N}_{t,k}$ and $u^{1:N}_{t,k}$
with $\hat{x}^{1:N}_{t,k}$ and $\hat{u}^{1:N}_{t,k}$, respectively, and equip
$\ess\times\eas/\sim$ with the quotient space
topology such that the natural projection $(x^{1:N}_{t,k},u^{1:N}_{t,k})\mapsto
(\hat{x}^{1:N}_{t,k},\hat{u}^{1:N}_{t,k})$ is rendered continuous. Notice also
that the map
\begin{align}
  (x^{1:N}_{t,k},u^{1:N}_{t,k})\longmapsto (\mu^{1:M}_{t,k},\theta^{1:M}_{t,k})=(\mu^{1:M}[x^{1:N}_{t,k}],\mu^{1:M}[(x^{1:N}_{t,k},u^{1:N}_{t,k})])\label{eq24}
\end{align}
is a (weakly) continuous surjection, since if $j\in V$, $g^j_1\in C_b(\mbb{X}^j)$ and $g^j_2\in
C_b(\mbb{X}^j\times\mbb{U}^j)$, then both
\begin{align*}
  \int_{\mbb{X}^j}g^j_1(x)\mu^j[x^{1:N}_{k}](dx)&=\frac{1}{N_j}\sum_{i\in\mc{C}_j}g^j_1(x^i_{k})\\
  &\overset{k\rightarrow\infty}{\longrightarrow}\frac{1}{N_j}\sum_{i\in\mc{C}_j}g^j_1(x^i)=\int_{\mbb{X}^j}g^j_1(x)\mu^{j}[x^{1:N}](dx),\quad\text{and}\\
  \int_{\mbb{X}^j\times\mbb{U}^j}g^j_2(x,u)\mu^j[(x^{1:N}_{k},u^{1:N}_{k})]&=\frac{1}{N_j}\sum_{i\in\mc{C}_j}g^j_2(x^i_{k},u^i_{k})\\
  &\overset{k\rightarrow\infty}{\longrightarrow}\frac{1}{N_j}\sum_{i\in\mc{C}_j}g^j_2(x^i,u^i)
  =\int_{\mbb{X}^j\times\mbb{U}^j}g^j_2(x,u)\mu^j[(x^{1:N},u^{1:N})](dx\times du)
\end{align*}
 whenever $(x^{1:N}_k,u^{1:N}_k)\rightarrow (x^{1:N},u^{1:N})$
as $k\rightarrow\infty$ in $\ess\times\eas$. Surjectivity is clear from the definitions of $\wh{\mbb{X}}^N$, $\wh{\mbb{U}}^N$. This makes (\ref{eq24}) a continuous
quotient mapping, implying that the spaces $\wh{\mbb{U}}^N$ and
$\ess\times\eas/\sim$ are homeomorphic (for details
on this, see \cite[theorem 4.2]{Armstrong_1983}, for example). Therefore,
$\theta^{1:M}_{t,k}\rightarrow\theta^{1:M}_{t}$ implies that the sequence of
representatives $(\hat{x}^{1:N}_{t,k},\hat{u}^{1:N}_{t,k})$ converges to some
$(\hat{x}^{1:N}_t,\hat{u}^{1:N}_{t})\in
\ess\times\eas/\sim$ as $k\rightarrow\infty$.\\[5pt]
\indent Now, fixing $g\in C_b(\wh{\mbb{X}}^N)$, we compute
\begin{align}
  \int_{\wh{\mbb{X}}^N}g(\mu)\wh{\mc{T}}(d\mu|\mu^{1:M}_{t,k},\theta^{1:M}_{t,k})&=\int_{\wh{\mbb{X}}^N}g(\mu)\int_{\prod_{i\in\mc{N}}\mbb{W}^{C(i)}}\delta_{\mu^{1:M}[F^N(\hat{x}^{1:N}_{t,k},\hat{u}^{1:N}_{t,k},\mu^{1:M}_{t,k},w^{1:N}_t)]}(d\mu)\nu(dw^{1:N}_t)\notag\\
  &=\int_{\prod_{i\in\mc{N}}\mbb{W}^{C(i)}}g(\mu^{1:M}[F^N(\hat{x}^{1:N}_{t,k},\hat{u}^{1:N}_{t,k},\mu^{1:M}_{t,k},w^{1:N}_t)])\nu(dw^{1:N}_t)\label{eq18}\\
  &\overset{k\rightarrow\infty}{\longrightarrow}\int_{\prod_{i\in\mc{N}}\mbb{W}^{C(i)}}g(\mu^{1:M}[F^N(\hat{x}^{1:N}_{t},\hat{u}^{1:N}_{t},\mu^{1:M}_{t},w^{1:N}_t)])\nu(dw^{1:N}_t)\label{eq19}\\
  &=\int_{\wh{\mbb{X}}^N}g(\mu)\wh{\mc{T}}^N(d\mu|\mu^{1:M}_{t},\theta^{1:M}_{t})\label{eq20}
\end{align}
where each line is obtained as follows. Equality (\ref{eq18}) follows the
Fubini-Tonelli theorem, (\ref{eq19}) the dominated convergence theorem, since
$g$ is continuous and bounded, and $(\mu^{1:M}\circ
f^{1:N})(\cdot,\cdot,\cdot,w^{1:N}_t)$ is jointly continuous on
$\ess\times\eas\times\mc{P}(\mbb{X}^1)\times\mc{P}(\mbb{X}^2)\times\cdots\times\mc{P}(\mbb{X}^M)$ for any $w^{1:N}_t\in\prod_{i\in\mc{N}}\mbb{W}^{C(i)}$ fixed
(since $F^N$ is jointly continuous under assumption \ref{ass1} and $\mu^{1:M}[\cdot]$
is continuous following the continuity of (\ref{eq24})).
Lastly, (\ref{eq20}) is obtained by reversing the steps used to obtain
(\ref{eq18}) from the first expression. Thus, $\wh{\mc{T}}^N$ is weak Feller as
claimed.\\[5pt]
\indent The stagewise cost $\wh{c}$ is continuous, since
\begin{align}
  \lim_{k\rightarrow\infty}\wh{c}(\mu^{1:M}_{t,k},\theta^{1:M}_{t,k})&=\lim_{k\rightarrow\infty}\sum_{j\in V}\int_{\mbb{X}^j\times\mbb{U}^j}c^j(x^j,u^j,\mu^{1:M}_{t,k})\theta^j_{t,k}(dx^j\times du^j)\label{eq21}\\
  &=\lim_{k\rightarrow\infty}\sum_{j\in V}\frac{1}{N_j}\sum_{i\in\mc{C}_j}c^j(\hat{x}^i_{t,k},\hat{u}^i_{t,k},\mu^{1:M}[\hat{x}^{1:N}_{t,k}])\label{eq22}\\
  &=\sum_{j\in V}\frac{1}{N_j}\sum_{i\in\mc{C}_j}c^j(\hat{x}^i_t,\hat{u}^i_t,\mu[\hat{x}^{1:N}_t])\label{eq23}\\
  &=\wh{c}(\mu^{1:M}_t,\theta^{1:M}_t)\notag
\end{align}
where (\ref{eq21}) is the definition of $\wh{c}$ from lemma \ref{lem1},
(\ref{eq22}) uses the same convergent sequence of representatives for
$\mu^{1:M}_{t,k}$ and $\theta^{1:M}_{t,k}$ as before, and (\ref{eq23}) follows the joint continuity of
each $c^j$ from assumption \ref{ass1}.\\[5pt]
\indent Finally, we establish compactness and upper semicontinuity of the action constraint. For the former, simply observe
that with $x^{1:N}\in\ess$, the set $\{x^{1:N}\}\times\eas$ is compact
in $\ess\times\eas$, the map in (\ref{eq24}) is continuous, and the
image of $\{x^{1:N}\}\times\eas$ under this map is precisely $\wh{\mbb{U}}^N(\mu^{1:M}[x^{1:N}])$. For upper semicontinuity, let $\mu^{1:M}_{k}\in\wh{\mbb{X}}^N$ and $\theta^{1:M}_{k}\in\wh{\mbb{U}}^N(\mu^{1:M}_{k})$
for $k\geq 1$, and further suppose that $\mu^{1:M}_k\rightarrow\mu^{1:M}\in\wh{\mbb{X}}^N$ as $k\rightarrow\infty$. From $\{\theta^{1:M}_k\}_{k\geq 1}$, we can construct a corresponding sequence $\{(x^{1:N}_k,u^{1:N}_k)\}_{k\geq 1}\subseteq\ess\times\eas$ so that
$\theta^{1:M}_k=\mu^{1:M}[(x^{1:N}_k,u^{1:N}_k)]$ for $k\geq 1$. 
Assumption \ref{ass1} implies $\ess\times\eas$ is compact, so there exists a convergent subsequence $(x^{1:N}_{k_\ell},u^{1:N}_{k_\ell})\rightarrow (x^{1:N},u^{1:N})$ as $\ell\rightarrow\infty$
such that for $j\in V$, $g^j\in C_b(\mbb{X}^j\times\mbb{U}^j)$ we have
\begin{align*}
  \int_{\mbb{X}^j\times\mbb{U}^j}g^j(x,u)&\mu^j[(x^{1:N}_{k_\ell},u^{1:N}_{k_\ell})](dx\times du)=\frac{1}{N_j}\sum_{i\in\mc{C}_j}g^j(x^i_{k_\ell},u^i_{k_\ell})\\
  &\overset{\ell\rightarrow\infty}{\longrightarrow}\frac{1}{N_j}\sum_{i\in\mc{C}_j}g^j(x^i,u^i)=\int_{\mbb{X}^j\times\mbb{U}^j}g^j(x,u)\mu^j[(x^{1:N},u^{1:N})](dx\times du)
\end{align*}
so the corresponding subsequence of actions $\theta^{1:M}_{k_\ell}$ converges weakly to another $\theta^{1:M}$ with $\mbb{X}^j$-marginals agreeing with the elements of $\mu^{1:M}$. That is,
the subsequential limit satisfies $\theta^{1:M}\in\wh{\mbb{U}}^N(\mu^{1:M})$, so the action constraint is upper semicontinuous.\hfill{$\qed$}

\section{Extension of \texorpdfstring{$N$}{TEXT}-Cluster-Exchangeable Random Variables}\label{app2}
In this appendix, we provide a detailed proof of Theorem \ref{thm8}\\[5pt]
\noindent\textbf{Proof}. Fix $1\leq k_j\leq N_j$ and let
$I^1_j,I^2_j,\dots\overset{\iid}{\sim}\mc{U}(\{1,2,\dots,N_j\})$ for $j\in V$.
We first establish a standard inequality for the total variation distance: If $Y$ is a $\mbb{Y}$-valued random variable on a
probability space $(\Omega,\mc{F},P)$ and $B\in\mc{F}$, 
then $\|\mc{L}(Y)-\mc{L}(Y|B)\|_{\text{TV}}\leq (1-P(B))$. To see this, compute
\begin{align*}
 \|\mc{L}(X)-\mc{L}(X|B)\|_{\text{TV}}&=\sup_{A\in\mc{B}(\mbb{Y})}\left|P(Y^{-1}(A))-\frac{P(Y^{-1}(A)\cap B)}{P(B)}\right|\\
 &\leq\sup_{A\in\mc{B}(\mbb{Y})}\left(P(Y^{-1}(A)\cap B^c)+\frac{1-P(B)}{P(B)}P(Y^{-1}(A)\cap B)\right)
\end{align*}
using the law of total probability followed by the triangle inequality. Setting
$\xi_A:=P(Y^{-1}(A))$ and $\zeta_A:=P(Y^{-1}(A)\cap B)$ this can be expressed as
a convex combination
\begin{align}
  \sup_{A\in\mc{B}(\mbb{Y})}\bigg(P(Y^{-1}(A)\cap B^c)&+\frac{1-P(B)}{P(B)}P(Y^{-1}(A)\cap B)\bigg)\notag\\
  &=\sup_{A\in\mc{B}(\mbb{Y})}\xi_A\left(\frac{\xi_A-\zeta_A}{\xi_A}+\frac{1-P(B)}{P(B)}\frac{\zeta_A}{\xi_A}\right)\notag\\
  &\leq\sup_{A\in\mc{B}(\mbb{Y})}\max\left\{\xi_A\1{P(Y^{-1}(A)\cap B)=0},\xi_A\frac{1-P(B)}{P(B)}\1{P(Y^{-1}(A)\cap B^c)=0}\right\}\label{eq49}\\
  &\leq\sup_{A\in\mc{F}}\max\left\{P(A)\1{P(A\cap B)=0},P(A)\frac{1-P(B)}{P(B)}\1{P(A\cap B^c)=0}\right\}\label{eq50}
\end{align}
where (\ref{eq49}) holds since $(\xi_A-\zeta_A)/\xi_A=1$ precisely when $P(Y^{-1}(A)\cap B)=0$ and
$\zeta_A/\xi_A=1$ if and only if $P(Y^{-1}(A)\cap B^c)=0$, and (\ref{eq50}) is taking the supremum over a possibly finer $\sigma$-field.
This supremum is attained by either $B^c$ (for the first maximand) or $B$ (for the second) but produces the desired bound in each case.\\[5pt]
\indent For the main claim, define $Y^i:=\bar{Y}^{I_{C(i)}^i}$ for $i\in\mbb{N}$, whereby $Y^{\mc{C}^\infty_j}\overset{\iid}{\sim}\mu^j[\bar{Y}^{1:N}]$.
This implies $(Y^{\mc{C}^\infty_1},Y^{\mc{C}_2^\infty},\dots, Y^{\mc{C}^\infty_M})$ is
C-Ex. To see this, form finite subclusters
$\mc{C}^\prime_j$ comprised of the first $k_j$ elements of $\mc{C}_j^\infty$ for
$j\in V$. Let $\mc{C}^\prime$ denote the partition giving rise to this population.
Then, with $\sigma\in\mc{S}_{\mc{C}^\prime}$ and $A^j\in\mc{B}(\prod_{i=1}^{k_j}\mbb{Y}^j)$
for $j\in V$,
\begin{align*}
  &P(Y^{\sigma(\mc{C}^\prime_1)}\in A^1,Y^{\sigma(\mc{C}_2^\prime)}\in A^2,\dots,Y^{\sigma(\mc{C}^\prime_M)}\in A^M)\\
  &=\sum_{i^{\mc{C}^\prime_j}_j\subseteq\mc{C}^\prime_j,\;j\in V}P(Y^{i^{\sigma(\mc{C}^\prime_1)}_1}\in A^1,\dots, Y^{i^{\sigma(\mc{C}^\prime_M)}_M}\in A^K)P(I^{\mc{C}^\prime_1}_{1}=i^{\mc{C}^\prime_1}_1,\dots, I^{\mc{C}^\prime_M}_M=i^{\mc{C}^\prime_M}_M)\\
  &=\sum_{i^{\mc{C}^\prime_j}_j\subseteq\mc{C}^\prime_j,\;j\in V}P(Y^{i^{\mc{C}^\prime_1}_1}\in A^1,\dots, Y^{i^{\mc{C}^\prime_M}_M}\in A^K)P(I^{\mc{C}^\prime_1}_{1}=i^{\sigma(\mc{C}^\prime_1)}_1,\dots, I^{\mc{C}^\prime_M}_M=i^{\sigma(\mc{C}^\prime_M)}_M)\\
  &=\sum_{i^{\mc{C}^\prime_j}_j\subseteq\mc{C}^\prime_j,\;j\in V}P(Y^{i^{\mc{C}^\prime_1}_1}\in A^1,\dots, Y^{i^{\mc{C}^\prime_M}_M}\in A^K)P(I^{\mc{C}^\prime_1}_{1}=i^{\mc{C}^\prime_1}_1,\dots, I^{\mc{C}^\prime_M}_M=i^{\mc{C}^\prime_M}_M)\\
  &=P(Y^{\mc{C}^\prime_1}\in A^1,Y^{\mc{C}_2^\prime}\in A^2,\dots,Y^{\mc{C}^\prime_M}\in A^M)
\end{align*}
where the third equality holds due to the $I^{\mc{C}^\prime_j}_j$ being $\iid$
(and hence exchangeable) for each $j\in V$.\\[5pt]
\indent Finally, consider the event $\beta^{k_j,N_j}_j$ that
the $I^{\mc{C}^\prime_j}_j$ are pairwise distinct. Clearly,
$P(\beta^{k_j,N_j}_j)=\prod_{k=1}^{k_j-1}(1-k/N_j)$, and
$\mc{L}(Y^{\mc{C}^\prime_1},Y^{\mc{C}^\prime_2},\dots,Y^{\mc{C}^\prime_M}|\beta^{k_1,N_1}_1,\beta^{k_2,N_2}_2,\dots,\beta^{k_M,N_M}_M)=\mc{L}(\bar{Y}^{\mc{C}^\prime_1},\bar{Y}^{\mc{C}^\prime_2},\dots,\bar{Y}^{\mc{C}^\prime_M})$
due to the cluster-exchangeability of $\bar{Y}^{1:N}$. Using our previous estimate, we
obtain
\begin{align*}
  \|\mc{L}(\bar{Y}^{\mc{C}^\prime_1},\bar{Y}^{\mc{C}^\prime_2},&\dots, \bar{Y}^{\mc{C}^\prime_M})-\mc{L}(Y^{\mc{C}^\prime_1},Y^{\mc{C}^\prime_2},\dots, Y^{\mc{C}^\prime_M})\|_{\text{TV}}\\
  &=\|\mc{L}(\bar{Y}^{\mc{C}^\prime_1},\bar{Y}^{\mc{C}^\prime_2},\dots,\bar{Y}^{\mc{C}^\prime_M}|\beta^{k_1,N_1}_1,\beta^{k_2,N_2}_2,\dots,\beta^{k_M,N_M}_M)-\mc{L}(\bar{Y}^{\mc{C}^\prime_1},\bar{Y}^{\mc{C}^\prime_2},\dots,\bar{Y}^{\mc{C}^\prime_M})\|_{\text{TV}}\\
  &\leq 1-\prod_{j\in V}\prod_{k=1}^{k_j-1}\left(1-\frac{k}{N_j}\right)
  \leq 1-\prod_{j\in V}\left(1-\sum_{k=1}^{k_j-1}\frac{k}{N_j}\right)
  = 1-\prod_{j\in V}\left(1-\frac{k_j(k_j-1)}{2N_j}\right)
\end{align*}
where the penultimate expression is easily obtained by induction on the
truncations $k_j$.\hfill{$\qed$}

\bibliography{graphon-1}
\bibliographystyle{abbrv}
\end{document}